\documentclass[12pt,a4paper,reqno]{article}%
\usepackage{amssymb}
\usepackage{amsmath}
\usepackage{amsfonts}
\usepackage{graphicx}
\usepackage{rotating}
\usepackage{mathrsfs}
\usepackage[all]{xy}
\usepackage{hyperref}%
\providecommand{\U}[1]{\protect\rule{.1in}{.1in}}
\SelectTips{cm}{}
\newtheorem{theorem}{Theorem}
\newtheorem{lemma}[theorem]{Lemma}
\newtheorem{remark}{Remark}
\newtheorem{definition}{Definition}
\newtheorem{example}{Example}
\newtheorem{proposition}[theorem]{Proposition}
\newtheorem{corollary}[theorem]{Corollary}
\newenvironment{proof}[1][Proof]{\noindent\textbf{#1.} }{\ \rule{0.5em}{0.5em}}
\begin{document}

\title{\textbf{Hamilton-Jacobi Diffieties}}
\author{\textsc{{L.{} Vitagliano}\thanks{\textbf{e}-\textit{mail}:
\texttt{lvitagliano@unisa.it}}} \\ {} \\ {\small {DipMat, Universit\`a degli Studi di Salerno, and}} \\ {\small {Istituto Nazionale di Fisica Nucleare, GC Salerno}} \\ {\small {Via Ponte don Melillo, 84084 Fisciano (SA), Italy}} \\ {} \\ {\small {Istituto Tullio Levi-Civita}} \\ {\small {via Colacurcio 54, 83050 Santo Stefano del Sole (AV), Italy}} }
\maketitle


\begin{abstract}
Diffieties formalize geometrically the concept of differential equations. We introduce and study Hamilton-Jacobi diffieties. They are finite dimensional subdiffieties of a given diffiety and appear to play a special role in the field theoretic version of the geometric Hamilton-Jacobi theory.
\end{abstract}

\newpage

\section*{Introduction}

Diffieties \cite{v01,v98} are geometric objects formalizing, in a coordinate
free manner, the concept of (systems of) differential equations (much as
varieties formalize, in a coordinate free manner, the concept of algebraic
equations). Roughly speaking, a diffiety is a manifold $\mathscr{E}$, of
countable dimension, endowed with an involutive distribution $\mathscr{C}$ of
finite dimension. Let $(\mathscr{E},\mathscr{C})$ be a diffiety representing a
certain system of PDEs $\mathscr{E}_{0}$. Then, (locally) maximal integral
submanifolds of $(\mathscr{E},\mathscr{C})$ represent (local) solutions of
$\mathscr{E}_{0}$. Notice that the Frobenius theorem fails for generic
infinite dimensional diffieties. In fact, a PDE needs not to possess solutions
with given (admissible) jet or may possess many solutions with the same
(admissible) jet. A particularly simple class of PDEs is made of compatible
PDEs of the form
\begin{equation}
\dfrac{\partial\boldsymbol{y}}{\partial x^{i}}=f_{i}(\boldsymbol{x}%
,\boldsymbol{y}). \label{24}%
\end{equation}
They are represented by a particularly simple class of diffieties, namely,
finite dimensional ones. For such diffieties the Frobenius theorem holds and,
in fact, integrable PDEs of the form (\ref{24}) possess exactly one (germ of a)
solution for any (admissible) jet. Given any PDE $\mathscr{E}_{0}$, one can
search for another compatible PDE $\mathscr{Y}$ of the simple form (\ref{24})
that implies $\mathscr{E}_{0}$. One may then obtain some solutions of
$\mathscr{E}_{0}$ by integrating such a $\mathscr{Y}$. Geometrically, one can
search for finite dimensional subdiffieties of the diffiety
$(\mathscr{E},\mathscr{C})$ representing $\mathscr{E}_{0}$. In this paper we
name as \emph{Hamilton-Jacobi diffieties} such subdiffieties and study their main
properties. The choice of this specific name is motivated by the following example.

Consider a system of ordinary (not necessarily autonomous) Euler--Lagrange
equations
\begin{equation}
\dfrac{\partial L}{\partial x^{i}}-\dfrac{d}{dt}\dfrac{\partial L}%
{\partial\dot{x}^{i}}=0,\quad i=1,\ldots,n. \label{OEL}%
\end{equation}
Suppose, for simplicity, that the Legendre transform
\[
p_{i}=\dfrac{\partial L}{\partial\dot{x}^{i}}(t,x,\dot{x})
\]
is invertible and let
\[
\dot{x}^{i}=v^{i}(t,x,p)
\]
be its inverse. Equation (\ref{OEL}) is represented by a suitable diffiety
$\mathscr{E}_{EL}$. Let
\begin{equation}
\tfrac{d}{dt}x^{i}=X^{i}(t,x) \label{26}%
\end{equation}
be a 1st order equation. It is of the form (\ref{24}) and, since it has just
one independent variable, it is automatically compatible. Therefore, it is
represented by a finite dimensional diffiety $\mathscr{Y}$. Moreover,
$\mathscr{Y}\subset\mathscr{E}_{EL}$, i.e., solutions of Eq.{} (\ref{26})
are solutions of Eq.{} (\ref{OEL}), iff
\begin{equation}
\dfrac{\partial L}{\partial x^{i}}(t,x,X)-\dfrac{\partial^{2}L}{\partial
\dot{x}^{i}\partial x^{j}}(t,x,X)X^{j}-\dfrac{\partial^{2}L}{\partial\dot
{x}^{i}\partial\dot{x}^{j}}(t,x,X)\left(  \dfrac{\partial X^{j}}{\partial
t}+X^{k}\dfrac{\partial X^{j}}{\partial x^{k}}\right)  =0 \label{15}%
\end{equation}
Notice that if $S=S(t,x)$ is a solution of the standard Hamilton-Jacobi (HJ)
equation
\begin{equation}
\dfrac{\partial S}{\partial t}+H(t,x,\partial S/\partial x)=f(t),\quad
H=v^{i}\dfrac{\partial L}{\partial\dot{x}^{i}}(t,x,v)-L(t,x,v), \label{14}%
\end{equation}
then
\[
X^{i}=v^{i}(t,x,\partial S/\partial x)
\]
is a solution of Eq.{} (\ref{15}). On the other hand, let $X^{i}%
=X^{i}(t,x)$ be a solution of Eq.{} (\ref{15}). Put
\[
T_{i}=\dfrac{\partial L}{\partial\dot{x}^{i}}(t,x,X).
\]
If
\begin{equation}
\dfrac{\partial T_{j}}{\partial x^{i}}-\dfrac{\partial T_{i}}{\partial x^{j}%
}=0, \label{16}%
\end{equation}
then $T_{i}=\partial S/\partial x^{i}$ for a solution $S=S(t,x)$ of the
standard HJ Eq.{} (\ref{14}). We conclude that the standard HJ equation is
basically equivalent to Eq.{} (\ref{15}) plus Eq.{} (\ref{16}).
Therefore, we name Eq.{} (\ref{15}) the \emph{generalized HJ equation} (see
below; see also \cite{c...06,c...09}). According to the above definition, its
solutions correspond to \emph{HJ subdiffieties} of $\mathscr{E}_{EL}$. This
motivates the choice of the name for these specific diffieties.

Similarly, under suitable integrability conditions, solutions of the field
theoretic HJ Eq.{} \cite{r73} may be interpreted as finite dimensional
subdiffieties of the diffiety of field equations. Hence, both the ordinary and the field
theoretic HJ theories fit well within the theory of HJ diffieties. In fact, we
show below that the geometric HJ theory presented in \cite{c...06,c...09} (see
also \cite{c...10} for the nonholonomic case) and generalized to the case of
regular field theories in \cite{v10}, can be naturally generalized to the case
of singular (i.e., gauge), higher derivative, field theories. In this
generalization, HJ diffieties play a central role.

The paper is divided into seven sections. In Section \ref{SecNotConv} we review
the basic differential geometric constructions used throughout the paper, and
collect notation and conventions. In Section \ref{SecGeomPDE} we recall the
concept of diffiety and briefly review the geometric theory of PDEs and its
application to the calculus of variations. In Section \ref{SecHJDiff} we
introduce the concept of HJ diffieties and elementary HJ diffieties, and study
their relation. In Section \ref{SecHJDiffEx} we illustrate the general theory
by presenting some simple examples of (elementary) HJ subdiffieties of
noteworthy diffieties. In Section \ref{SecLHFT} we review the geometric
formulation of higher derivative, Lagrangian and Hamiltonian field theories as
defined in \cite{v10b}. In Section \ref{SecHJEL} we propose a field theoretic
version of the geometric HJ theory of \cite{c...06,c...09} and show that it is
naturally linked to the theory of (elementary) HJ subdiffieties of the field
equations. In Section \ref{SecFinEx} we present one final example of
(elementary) HJ subdiffieties of an Euler--Lagrange equation. The review
Sections \ref{SecNotConv}, \ref{SecGeomPDE} and \ref{SecLHFT} are included to
make the paper as self-consistent as possible.

\section{Differential Geometric Background\label{SecNotConv}}

In this section we collect notation, conventions, and the main geometric
constructions needed in the paper.

Let $N$ be a smooth manifold. If $L\subset N$ is a submanifold, we denote the inclusion by
$i_{L}:L\hookrightarrow N$. We denote by $C^{\infty}(N)$ the
$\mathbb{R}$-algebra of smooth, $\mathbb{R}$-valued functions on $N$. We
will always understand a vector field $X$ on $N$ as a derivation $X:C^{\infty
}(N)\longrightarrow C^{\infty}(N)$. We denote by $D(N)$ the $C^{\infty}%
(N)$-module of vector fields over $N$, by $\Lambda(M)=\bigoplus_{k}%
\Lambda^{k}(N)$ the graded $\mathbb{R}$-algebra of differential forms over
$N$ and by $d:\Lambda(N)\longrightarrow\Lambda(N)$ the de Rham differential.
If $F:N_{1}\longrightarrow N$ is a smooth map of manifolds, we denote by
$F^{\ast}:\Lambda(N)\longrightarrow\Lambda(N_{1})$ the pull-back via $F$.

Let $\alpha:A\longrightarrow N$ be an affine bundle (for instance, a vector
bundle) and $F:N_{1}\longrightarrow N$ a smooth map of manifolds. Let
$\mathscr{A}$ be the affine space of smooth sections of $\alpha$. The affine
bundle on $N_{1}$ induced by $\alpha$ via $F$ will be denoted by $F^{\circ
}(\alpha):F^{\circ}(A)\longrightarrow N$:
\[%
\begin{array}
[c]{c}%
\xymatrix{F^\circ(A) \ar[r] \ar[d]_-{F^\circ(\alpha)} & A \ar[d]^-{\alpha} \\ N_1 \ar[r]^-F & N }
\end{array}
,
\]
and the space of its section by $F^{\circ}(\mathscr{A})$. For any section $a$
of $\alpha$ there exists a unique section of $F^{\circ}(\sigma)$, which we
denote by $F^{\circ}(a)$, such that the diagram
\[
\xymatrix{F^\circ(A) \ar[r]  & A  \\
N_1 \ar[r]^-F    \ar[u]^-{F^\circ(a)}                    &  N \ar[u]_-{a}}
\]
commutes. If $F:N_{1}\longrightarrow N$ is the embedding of a submanifold, we
also write $\bullet\ |_{F}$ for $F^{\circ}({}\bullet{})$.

We will often understand the sum over repeated upper-lower indices and
multi-indices. Our notation as regards multi-indices is as follows. We will use
the capital letters $I,J,K$ for multi-indices. Let $n$ be a positive integer. A
multi-index of length $k$ is a $k$tuple of indices $I=(i_{1},\ldots,i_{k})$,
$i_{1},\ldots,i_{k}\leq n$. We identify multi-indices differing only by the
order of the entries. If $I$ is a multi-index of length $k$, we put $|I|:=k$.
Let $I=(i_{1},\ldots,i_{k})$ and $J=(j_{1},\ldots,j_{h})$ be multi-indices, and
$i$ be an index. We denote by $IJ$ (resp.{} $Ii$) the multi-index $(i_{1}%
,\ldots,i_{k},j_{1},\ldots,j_{h})$ (resp.{} $(i_{1},\ldots,i_{k},i)$). We
write $\partial^{|I|}/\partial x^{I}$ for $\partial/\partial x^{i_{1}}%
\circ\cdots\circ\partial/\partial x^{i_{k}}$.

Let $\xi:P\longrightarrow M$ be a fiber bundle. For $k\leq\infty$, we denote
by $\xi_{k}:J^{k}\xi\longrightarrow M$ the bundle of $k$-jets of local
sections of $\xi$. For any (local) section $s:M\longrightarrow P$ of $\pi$, we
denote by $j_{k}s:M\longrightarrow J^{k}\xi$ its $k$th jet prolongation. Let
$\ldots,x^{i},\ldots$ be coordinates on $M$ and $\ldots,x^{i},\ldots
,y^{a},\ldots$ bundle coordinates on $E$. We denote by $\ldots,x^{i}%
,\ldots,y_{I}^{a},\ldots$ the associated jet coordinates on $J^{k}\xi$,
$|I|{}\leq k$. For $0\leq h\leq k\leq\infty$, we denote by $\xi_{k,h}:J^{k}%
\xi\longrightarrow J^{h}\xi$ the canonical projection. We will always
understand the monomorphisms $\pi_{k,h}^{\ast}:\Lambda(J^{h}\xi
)\longrightarrow\Lambda(J^{k}\xi)$. For all $k\geq0$, $\xi_{k+1,k}:J^{k+1}%
\xi\longrightarrow J^{k}\xi$ is an affine subbundle of $(\pi_{k})_{1,0}%
:J^{1}\pi_{k}\longrightarrow J^{k}\pi$ and the inclusion $i:J^{k+1}\pi\subset
J^{1}\pi_{k}$ is locally defined by $i^{\ast}(u_{I}^{\alpha})_{i}%
=u_{Ii}^{\alpha}$, $|I|{}\leq k$.

Let $\xi$ be as above. We denote by $\Lambda_{1}(P,\xi)=\bigoplus_{k}%
\Lambda_{1}^{k}(P,\xi)\subset\Lambda(P)$ the differential (graded) ideal in
$\Lambda(P)$ made of differential forms on $P$ vanishing when pulled-back to
fibers of $\xi$, by $\Lambda_{q}(P,\xi)=\bigoplus_{k}\Lambda_{q}^{k}(P,\xi)$
its $q$th exterior power, $q\geq0$, and by $V\!\Lambda(P,\xi)=\bigoplus
_{k}V\!\Lambda^{k}(P,\xi)$ the quotient differential algebra $\Lambda
(P)/\Lambda_{1}(P,\xi)$, $d^{V}:V\!\Lambda(P,\xi)\longrightarrow
V\!\Lambda(P,\xi)$ being its (quotient) differential. By abusing the notation,
we also denote by $d^{V}$ the (quotient) differential in $\Lambda_{q}%
(P,\xi)/\Lambda_{q+1}(P,\xi)\simeq V\!\Lambda(P,\xi)\otimes\Lambda_{q}%
^{q}(P,\xi)$. We denote by $J^{\dag}\xi\longrightarrow P$ the \emph{reduced
multimomentum bundle} of $\xi$ (see, for instance, \cite{r05}). It is the
vector bundle over $P$ whose module of sections is $V\!\Lambda^{1}%
(P,\xi)\otimes\Lambda_{n-1}^{n-1}(P,\xi)$. Equivalently, it may be defined as
the bundle of affine morphisms from $J^{1}\xi$ to $\Lambda^{n}T^{\ast}M$.

A connection $\nabla$ in $\xi$ is a section of the first jet bundle $\xi
_{1,0}:J^{1}\xi\longrightarrow P$. We will also interpret $\nabla$ as an
element in $\Lambda^{1}(P)\otimes V\!D(P,\xi)$, where $V\!D(P,\xi)$ is the
module of $\xi$-vertical vector fields on $P$. Put $\ldots,\nabla_{i}%
^{a}:=\nabla^{\ast}(y_{i}^{a}),\ldots$, where $\ldots,y_{i}^{a},\ldots$ are
jet coordinates in $J^{1}\xi$. Then, locally
\[
\nabla=(dy^{a}-\nabla_{i}^{a}dx^{i})\otimes\dfrac{\partial}{\partial y^{a}}.
\]
Recall that a (local) section $\sigma:M\longrightarrow P$ is $\nabla$-constant
for a connection $\nabla$ iff, by definition, $\nabla\circ\sigma=j_{1}\sigma$.
A connection $\nabla$ in $P$ determines splittings of the exact sequence
\begin{equation}
0\longrightarrow V\!D(P,\xi)\longrightarrow D(P)\longrightarrow\xi^{\circ
}(D(M))\longrightarrow0, \label{ExSeq}%
\end{equation}
and its dual%
\begin{equation}
0\longleftarrow V\!\Lambda^{1}(P,\xi)\longleftarrow\Lambda^{1}%
(P)\longleftarrow\Lambda_{1}^{1}(P,\xi)\longleftarrow0. \label{ExSeqDual}%
\end{equation}
Thus, using $\nabla$ one can lift a vector field $X$ on $M$ to a vector field
$X^{\nabla}$ on $P$ transversal to fibers of $\xi$. Moreover, $\nabla$ determines an
isomorphism
\[
\Lambda(P)\simeq\bigoplus_{p,q}V\!\Lambda^{p}(P,\xi)\otimes\Lambda_{q}%
^{q}(P,\xi),
\]
and, in particular, for any $p,q$, a projection
\[
i^{p,q}(\nabla):\Lambda^{p+q}(P)\longrightarrow V\!\Lambda^{p}(P,\xi
)\otimes\Lambda_{q}^{q}(P,\xi),
\]
and an embedding
\[
e^{p,q}(\nabla):V\!\Lambda^{p}(P,\xi)\otimes\Lambda_{q}^{q}(P,\xi
)\longrightarrow\Lambda^{p+q}(P)
\]
taking its values in $\Lambda_{q}^{p+q}(P,\xi)$. Notice that the
\textquotedblleft insertions\textquotedblright\ $i^{p,q}(\nabla)$ are actually
pointwise and, therefore, can be restricted to maps. Namely, if
$F:P_{1}\longrightarrow P$ is a smooth map, and $\Delta$ a section of the
pull-back $F^{\circ}(\xi_{1,0}):F^{\circ}(J^{1}\xi)\longrightarrow P_{1}$,
then the element
\[
i^{p,q}(\Delta)F^{\circ}(\omega)\in F^{\circ}(V\!\Lambda^{p}(P,\xi)\otimes
_{A}\Lambda_{q}^{q}(P,\xi))
\]
is well-defined for every $\omega\in\Lambda^{p+q}(P)$.

Every connection $\nabla$ defines a vector-valued differential $2$-form $R^{\nabla}\in\Lambda_{2}^{2}(P,\xi)\otimes V\!D(P,\xi)$, called the curvature, via
\[
R^{\nabla}(X,Y):=[X^{\nabla},Y^{\nabla}]-[X,Y]^{\nabla},\quad X,Y\in D(M)
\]
Locally,
\[
R^{\nabla}=R_{ij}^{a}dx^{i}\wedge dx^{j}\otimes\frac{\partial}{\partial y^{a}%
},\quad R_{ij}^{a}=\tfrac{1}{2}(D_{i}\nabla_{j}^{a}-D_{j}\nabla_{i}^{a}%
)\circ\nabla.
\]
where $D_{i}:=\partial/\partial x^{i}+y_{i}^{a}\partial/\partial y^{a}$,
$i=1,\ldots,n$. A connection $\nabla$ is flat iff, by definition, $R^{\nabla
}=0$. If $\nabla$ is a flat connection in $\xi$, then $P$ is locally foliated
by (local) $\nabla$-constant sections of $\xi$.

\begin{example}
Let $\xi:P\longrightarrow M$ be as above and $\sigma:M\longrightarrow P$ a
(local) section of $\xi$. It is sometimes useful to understand $j_{1}%
\sigma:M\longrightarrow J^{1}\xi$ as a section of the pull-back bundle
$\xi_{1,0}|_{\sigma}:J^{1}\xi|_{\sigma}\longrightarrow M$. For instance, if
$\omega\in\Lambda_{n-1}^{n+1}(P,\xi)$ is a PD Hamiltonian system in $\xi$ in
the sense of \cite{v09b}, the \emph{PD Hamilton equations} for $\sigma$ read
\[
i^{1,n}(j_{1}\sigma)\omega|_{\sigma}=0.
\]

\end{example}

Let $\nabla$ be a connection in $\xi$. If $\nabla$ is flat, the de Rham
complex of $P$, $(\Lambda(P),d)$, splits into a bicomplex
\begin{equation}
(V\!\Lambda(P,\xi)\otimes\Lambda_{\bullet}^{\bullet}(P,\xi),\overline{d}%
_{\nabla},d^{V}), \label{25}%
\end{equation}
where
\[
\overline{d}_{\nabla}(\omega\otimes\sigma):=(i^{p,q+1}(\nabla)\circ d\circ
e^{p,q}(\nabla))(\omega\otimes\sigma),
\]
$\omega\in V\!\Lambda^{p}(P,\xi)$ and $\sigma\in\Lambda_{q}^{q}(P,\xi)$,
$p,q\geq0$. For every fixed $p$, the complex $(V\!\Lambda^{p}(P,\xi
)\otimes\Lambda_{\bullet}^{\bullet}(P,\xi),\overline{d}_{\nabla})$ is locally
acyclic in positive degree.

\section{Geometry of PDEs and Calculus of Variations\label{SecGeomPDE}}

In this section we recall basic facts about the geometric theory of partial
differential equations (PDEs). For more details see \cite{b...99}.

Let $\pi:E\longrightarrow M$ be a fiber bundle, $\dim M=n$, $dim E = n+m$ and $\ldots
,u^{\alpha},\ldots$ fiber coordinates in $E$. Recall that, for all
$k\leq\infty$, $J^{k}\pi$ is endowed with the \emph{Cartan distribution}
\[
\mathscr{C}_{k}:J^{k}\pi\ni\theta\longmapsto\mathscr{C}_{k}(\theta)\subset
T_{\theta}J^{k}\pi,
\]
where $\mathscr{C}_{k}(\theta)$ is defined as follows. Suppose that $\theta
=(j_{k}s)(x)$ for some $x\in M$ and $s$ a local section of $\pi$ around $x$.
The image of $d_{x}j_{k}s:T_{x}M\longrightarrow T_{\theta}J^{k}\pi$ is said to
be an $R$\emph{-plane} at $\theta$. Put
\[
\mathscr{C}_{k}(\theta):=\operatorname{span}\{R\text{-planes at }\theta\}.
\]
Now, let $k<\infty$. $\mathscr{C}_{k}$ is locally spanned by vector fields
\[
\ldots,\dfrac{\partial}{\partial x^{i}}+\sum\nolimits_{|I|{}<k}u_{Ii}^{\alpha
}\dfrac{\partial}{\partial u_{I}^{\alpha}},\ldots,\dfrac{\partial}{\partial
u_{J}^{\alpha}},\ldots,\quad|J|{}=k.
\]

Local infinitesimal symmetries of $\mathscr{C}_{k}$ are called \emph{Lie
fields}. Every Lie field $Z\in D(J^{k}\pi)$ can be uniquely lifted to a Lie
field $Z_{r}\in D(J^{k+r}\pi)$. Moreover, according to the Lie-B\"{a}cklund Theorem, every Lie field $Z\in D(J^{k}\pi)$
is the lift of

\begin{enumerate}
\item a vector field $Y\in D(E)$ if $m>1$,

\item a Lie field $Y^{\prime}\in D(J^{1}\pi)$ if $m=1$.
\end{enumerate}

If $Y\in D(E)$ is locally given by $Y=X^{i}\partial/\partial x^{i}+Y^{\alpha
}\partial/\partial u^{\alpha}$, then the Lie field $Y_{r}\in D(J^{r}\pi)$ is
locally given by
\begin{equation}
Y_{r}=X^{i}\left(  \dfrac{\partial}{\partial x^{i}}+\sum\nolimits_{|I|{}%
<r}u_{Ii}^{\alpha}\dfrac{\partial}{\partial u_{I}^{\alpha}}\right)
+\sum\nolimits_{|I|{}<r}D_{I}(Y^{\alpha}-u_{i}^{\alpha}X^{i})\dfrac{\partial
}{\partial u_{I}^{\alpha}}, \label{21}%
\end{equation}
where $D_{j_{1}\cdots j_{s}}:=D_{j_{1}}\circ\cdots\circ D_{j_{s}}$, and
$D_{j}:=\partial/\partial x^{j}+u_{Ij}^{\alpha}\partial/\partial u_{I}%
^{\alpha}$ is the $j$\emph{th total derivative, }$j,j_{1},\ldots
,j_{s}=1,\ldots,n$.

\begin{remark}
\label{Remark3}The local expression (\ref{21}) shows, in particular, that, for
any $\xi\in T_{\theta}J^{r}\pi$, there exists $Y\in D(E)$ such that
$\xi=(Y_{r})_{\theta}$.
\end{remark}

A $k$th-order (system of $\ell$) PDE(s) on sections of $\pi$ is an $\ell
$-codimensional closed submanifold $\mathscr{E}_{0}\subset J^{k}\pi$. For
$0\leq r\leq\infty$ one can define the $r$\emph{th prolongation of }$E$ as
\begin{align*}
\mathscr{E}_{r}  &  :=\{j_{k+r}(s)(x)\in J^{k+r}\pi:\operatorname{im}%
j_{k}(s)\text{ is tangent to }\mathscr{E}_{0}\\
&  \text{at }j_{k}(s)(x)\text{ up to the order }r\text{, }x\in M\}\subset
J^{k+r}\pi.
\end{align*}
If $\mathscr{E}$ is locally given by
\[
\Phi^{a}(\ldots,x^{i},\ldots,u_{I}^{\alpha},\ldots)=0,\quad a=1,\ldots
,\ell,\quad|I|{}\leq k,
\]
then $\mathscr{E}_{r}$ is locally given by
\begin{equation}
(D_{J}\Phi^{a})(\ldots,x^{i},\ldots,u_{I}^{\alpha},\ldots)=0,\quad|J|{}\leq r,
\label{Prolong}%
\end{equation}
In the following we will always assume that

\begin{enumerate}
\item $\mathscr{E}_{r}\longrightarrow M$ is a smooth (closed) subbundle of
$\pi_{k+r}$, $r\leq\infty$,

\item the (possibly non-surjective) maps $\pi_{k+r+1,k+r}:\mathscr{E}_{r+1}%
\longrightarrow\mathscr{E}_{r}$, $r<\infty$, have constant rank.
\end{enumerate}

A local section $s$ of $\pi$ is a \emph{(local) solution of }$\mathscr{E}_{0}$
iff, by definition, $\operatorname{im}j_{k}s\subset\mathscr{E}_{0}$ or, which
is the same, $\operatorname{im}j_{k+r}s\subset\mathscr{E}_{r}$ for some
$r\leq\infty$. Lie fields in $D(J^{k}\pi)$ preserving $\mathscr{E}_{0}$ are
called \emph{Lie symmetries} of $\mathscr{E}_{0}$. The flow of a Lie symmetry
maps (images of $k$th prolongations of) solutions to (images of $k$th
prolongations of) solutions.

The Cartan distribution $\mathscr{C}:=\mathscr{C}_{\infty}\subset TJ^{\infty
}\pi$ is locally spanned by total derivatives, $D_{i}$, $i=1,\ldots,n$, and
restricts to any submanifold $\mathscr{E}\subset J^{\infty}\pi$ of the form
$\mathscr{E=E}_{\infty}$ for some PDE $\mathscr{E}_{0}$. Denote again by
$\mathscr{C}$ the restricted distribution. It is a flat connection in
$\mathscr{E}\longrightarrow M$ sometimes called the \emph{Cartan connection}.
$\mathscr{C}$-constant sections are of the form $j_{\infty}s$, with $s$ a
(local) solution of $\mathscr{E}_{0}$ and vice versa. Therefore, we can
identify the space of $\mathscr{C}$-constant sections and the space of solutions
of $\mathscr{E}_{0}$. The pair $(\mathscr{E},\mathscr{C})$ is called a(n
elementary) \emph{diffiety} \cite{v01,v98} and contains all the relevant
information about the original PDE $\mathscr{E}_{0}$. We will often identify
$(\mathscr{E},\mathscr{C})$ and $\mathscr{E}_{0}$.

\begin{remark}
\label{Remark4}A diffiety $(\mathscr{E},\mathscr{C})$ can be generically
embedded in many ways in an infinite jet space. Informally speaking, any such
embedding corresponds to a choice of dependent variables in the original
equation. Properties of a diffiety that do not depend on its embedding in an
infinite jet space are referred to as \emph{intrinsic}.
\end{remark}

In the following we will indicate $\mathscr{C}^{p}\Lambda^{p}:=V\!\Lambda
^{p}(J^{\infty}\pi,\pi_{\infty})$ and $\overline{\Lambda}{}^{q}:=\Lambda
_{q}^{q}(J^{\infty}\pi,\pi_{\infty})$. We also put $\mathscr{C}^{\bullet
}\Lambda:=\bigoplus_{p}\mathscr{C}^{p}\Lambda^{p}$ and $\overline{\Lambda
}:=\bigoplus_{q}\overline{\Lambda}{}^{q}$. The Cartan connection endows the de
Rham complex $(\Lambda(J^{\infty}\pi),d)$ of $J^{\infty}\pi$ with a bicomplex
structure $(\mathscr{C}^{\bullet}\Lambda\otimes\overline{\Lambda}{}%
,\overline{d},d^{V})$, where $\overline{d}:=\overline{d}_{\mathscr{C}}$,
called the \emph{variational bicomplex}${}$. The variational bicomplex allows
a cohomological formulation of the calculus of variations
\cite{b...99,v84,a92} (see below). Similarly, the de Rham complex
$(\Lambda(\mathscr{E}),d)$ of a diffiety $\mathscr{E}$ is naturally endowed
with a bicomplex structure denoted by $(\mathscr{C}^{\bullet}\Lambda
(\mathscr{E})\otimes\overline{\Lambda}{}(\mathscr{E}),\overline{d},d^{V})$.

In the following we will understand the isomorphism $\Lambda(J^{\infty}\pi%
)\simeq\mathscr{C}^{\bullet}\Lambda\otimes\overline{\Lambda}{}$. The complex
\[%
\xymatrix{0 \ar[r] &  C^\infty(J^\infty) \ar[r]^-{\overline{d}} & \overline
{\Lambda}{}^1 \ar[r]^-{\overline{d}}
& \cdots\ar[r] & \overline{\Lambda}{}^q \ar[r]^-{\overline{d}} & \overline
{\Lambda}{}^{q+1} \ar[r]^-{\overline{d}} & \cdots}%
\]
is called the \emph{horizontal de Rham complex}. An element $\mathscr{L}\in
\overline{\Lambda}{}^{n}$ is naturally interpreted as a \emph{Lagrangian
density} and its cohomology class $[\mathscr{L}]\in$ $H^{n}(\overline{\Lambda
},\overline{d})$ as an \emph{action functional} on sections of $\pi$. The
associated Euler--Lagrange equations can then be obtained as follows.

Consider the complex
\begin{equation}%
\xymatrix{0 \ar[r] & \mathscr{C}\Lambda^1 \ar[r]^-{\overline{d}} & \mathscr
{C}\Lambda^1\otimes\overline{\Lambda}{}^1 \ar[r]^-{\overline{d}} & \cdots
\ar[r] & \mathscr{C}\Lambda^1\otimes\overline{\Lambda}{}^q \ar[r]^-{\overline
{d}} & \cdots}%
, \label{ComplCL1}%
\end{equation}
and the $C^{\infty}(J^{\infty}\pi)$-submodule $\varkappa^{\dag}\subset
\mathscr{C}\Lambda^{1}\otimes\overline{\Lambda}{}^{n}$ generated by elements
in $\mathscr{C}\Lambda^{1}\otimes\overline{\Lambda}{}^{n}\cap\Lambda
^{n+1}(J^{1}\pi)$. $\varkappa^{\dag}$ is locally spanned by elements
$d^{V}\!u^{\alpha}\otimes d^{n}x$, $d^{n}x:=dx^{1}\wedge\cdots\wedge dx^{n}$.

\begin{theorem}
\label{Theorem0}\cite{v84} Complex (\ref{ComplCL1}) is acyclic in the $q$th
term, for $q\neq n$. Moreover, for any $\omega\in\mathscr{C}\Lambda^{1}%
\otimes\overline{\Lambda}{}^{n}$ there exists a unique element $\boldsymbol{E}%
_{\omega}\in\varkappa^{\dag}\subset\mathscr{C}\Lambda^{1}\otimes
\overline{\Lambda}{}^{q}$ such that $\boldsymbol{E}_{\omega}-\omega
=\overline{d}\vartheta$ for some $\vartheta\in\mathscr{C}\Lambda^{1}%
\otimes\overline{\Lambda}{}^{n-1}$ and the correspondence $H^{n}%
(\mathscr{C}\Lambda^{1}\otimes\overline{\Lambda},\overline{d})\ni\lbrack
\omega]\longmapsto\boldsymbol{E}_{\omega}\in\varkappa^{\dag}$ is a vector
space isomorphism. In particular, for $\omega=d^{V}\!\mathscr{L}$,
$\mathscr{L}\in\overline{\Lambda}{}^{n}$ being a Lagrangian density locally
given by $\mathscr{L}=Ld^{n}x$, $L$ a local function on $C^{\infty}(J^{\infty
}\pi)$, $\boldsymbol{E}(\mathscr{L}):=\boldsymbol{E}_{\omega}$ is locally given
by $\boldsymbol{E}(\mathscr{L})=\delta L/\delta u^{\alpha}d^{V}\!u^{\alpha
}\otimes d^{n}x$ where
\[
\frac{\delta L}{\delta u^{\alpha}}:=(-)^{|I|}D_{I}\frac{\partial L}{\partial
u_{I}^{\alpha}}%
\]
are the Euler--Lagrange derivatives of $L$.
\end{theorem}

In view of the above theorem, $\boldsymbol{E}(\mathscr{L})$ does not depend on
the choice of $\mathscr{L}$ in a cohomology class $[\mathscr{L}]\in$
$H^{n}(\overline{\Lambda},\overline{d})$ and it is naturally interpreted as
the left hand side of the Euler--Lagrange (EL) equations determined by
$\mathscr{L}$. Any $\vartheta\in\mathscr{C}\Lambda^{1}\otimes\overline
{\Lambda}{}^{n-1}$ such that
\begin{equation}
\boldsymbol{E}(\mathscr{L})-d^{V}\!\mathscr{L}=\overline{d}\vartheta
\label{Legendre}%
\end{equation}
will be called a \emph{Legendre form} \cite{av04}. Eq.{} (\ref{Legendre})
may be interpreted as the \emph{first variation formula} for the Lagrangian
density $\mathscr{L}$. A local Legendre form is given by
\[
\vartheta_{\mathrm{loc}}=(-)^{|J|}\binom{IJ}{J}D_{J}\dfrac{\partial
L}{\partial u_{IJi}^{\alpha}}d^{V}\!u_{I}^{\alpha}\otimes d^{n-1}x_{i},
\]
where $\binom{IJ}{J}$ is the multinomial coefficient, and $d^{n-1}x_i := i_{D_i} d^n x$.

\begin{remark}
\label{Remark2}Notice that, if $\vartheta\in\mathscr{C}\Lambda^{1}%
\otimes\overline{\Lambda}{}^{n-1}$ is a Legendre form for a Lagrangian density
$\mathscr{L}\in\overline{\Lambda}{}^{n}$, then $\vartheta+d^{V}\!\varrho$ is a
Legendre form for the cohomologous Lagrangian density $\mathscr{L}+\overline
{d}\varrho$, $\varrho\in\overline{\Lambda}{}^{n-1}$. Moreover, any two
Legendre forms $\vartheta,\vartheta^{\prime}$ for the same Lagrangian density
differ by a $\overline{d}$-closed, and, therefore, $\overline{d}$-exact form,
i.e., $\vartheta-\vartheta^{\prime}=\overline{d}\lambda$, for some $\lambda
\in\mathscr{C}\Lambda^{1}\otimes\overline{\Lambda}{}^{n-2}$. Finally, for
$\mathscr{L}\in\Lambda(J^{k+1}\pi)$ one can always find a Legendre form
$\vartheta$ containing vertical derivatives of functions in $C^{\infty}%
(J^{k}\pi)$ only.
\end{remark}

\begin{remark}
\label{Remark1}Complex (\ref{ComplCL1}) restricts to any diffiety
$\mathscr{E}$ in the sense that there is a (unique) complex
\begin{equation}
\xymatrix{0 \ar[r] & \mathscr{C}\Lambda^1|_\mathscr{E} \ar[r]^-{\overline{d}|_\mathscr{E}} & \mathscr{C}\Lambda^1\otimes\overline{\Lambda}{}^1|_\mathscr{E} \ar[r]^-{\overline{d}|_\mathscr{E}} & \cdots \ar[r] & \mathscr{C}\Lambda^1\otimes\overline{\Lambda}{}^q|_\mathscr{E} \ar[r]^-{\overline{d}|_\mathscr{E}} & \cdots},
\label{ComplCL1j}%
\end{equation}
such that the restriction map $\mathscr{C}\Lambda^{1}\otimes\overline{\Lambda
}{}\longrightarrow\mathscr{C}\Lambda^{1}\otimes\overline{\Lambda}%
{}|_{\mathscr{E}}$ is a morphism of complexes. Moreover, complex
(\ref{ComplCL1j}) is acyclic in the $q$th term and the correspondence defined
by $H^{n}(\mathscr{C}\Lambda^{1}\otimes\overline{\Lambda}{}^{n}|_{\mathscr{E}}%
,\overline{d}|_{\mathscr{E}})\ni\lbrack\omega|_{\mathscr{E}}]\longmapsto
\boldsymbol{E}_{\omega}|_{\mathscr{E}}\in\varkappa^{\dag}|_{\mathscr{E}}$,
$\omega\in\mathscr{C}\Lambda^{1}\otimes\overline{\Lambda}{}^{n}$, is a vector
space isomorphism.
\end{remark}

\section{Hamilton-Jacobi Diffieties\label{SecHJDiff}}

In the following we simply write $J^{k}$ for $J^{k}\pi$, $k\leq\infty$.

\begin{definition}
\label{Def2}Let $\mathscr{E}$ be a diffiety. A finite dimensional diffiety
$\mathscr{Y}$ will be called an \emph{Hamilton-Jacobi (HJ) diffiety}. If\emph{
}$\mathscr{Y}\subset\mathscr{E}$, then $\mathscr{Y}$ will be called an \emph{HJ
subdiffiety} of $\mathscr{E}$.
\end{definition}

Motivations for this definition can be found in the introduction and in
Section \ref{SecHJEL} (see also Example \ref{ExEL}). From an intrinsic point
of view an HJ diffiety is nothing but a (finite dimensional) manifold with an
involutive distribution, or, which is the same, a foliation. As recalled in
the introduction, the equation for the leaves of the foliation is a compatible
equation of the form (\ref{24}).

A diffiety $\mathscr{E}$ is often presented together with an embedding
$\mathscr{E}\subset J^{\infty}$. In this case, it is useful to restrict the
attention to a special class of HJ subdiffieties of $\mathscr{E}$ that we
define below.

A connection $\nabla:J^{k}\longrightarrow J^{1}\pi_{k}$ in $\pi_{k}%
:J^{k}\longrightarrow M$ will be said to be \emph{holonomic} if it takes its values
in $J^{k+1}\subset J^{1}\pi_{k}$. $\nabla$ is holonomic iff locally
\[%
\begin{array}
[c]{ll}%
(\nabla_{I}^{\alpha}{})_{i}=u_{Ii}^{\alpha}, & \text{if }|I|{}<k\\
(\nabla_{I}^{\alpha}{})_{i}=(\nabla_{J}^{\alpha}{})_{j}, & \text{if }%
|I|{}=|J|{}=k\text{ and }Ii=Jj
\end{array}
.
\]
Let $\nabla$ be an holonomic connection in $\pi_{k}$. For $|I|{}=k$, put
$\nabla_{Ii}^{\alpha}:=(\nabla_{I}^{\alpha}{})_{i}$. Since $\nabla$ is
holonomic, the $\nabla_{Ii}^{\alpha}$'s are well-defined. Moreover, local $\nabla$-constant sections are of the form $j_{k}s:M\longrightarrow J^{k}$, for some
local section $s$ of $\pi$. Notice that $\nabla$ is flat iff, locally,
\begin{equation}
\lbrack\nabla_{i},\nabla_{j}]=0, \label{9}%
\end{equation}
where%
\[
\nabla_{i}=\frac{\partial}{\partial x^{i}}+\sum\nolimits_{|I|{}<k}%
u_{Ii}^{\alpha}\frac{\partial}{\partial u_{I}^{\alpha}}+\sum\nolimits_{|I|{}%
=k}\nabla_{Ii}^{\alpha}\frac{\partial}{\partial u_{I}^{\alpha}}.
\]
Eq.{} (\ref{9}) can be rewritten as
\[
\nabla_{i}\nabla_{jI}^{\alpha}-\nabla_{j}\nabla_{iI}^{\alpha}=\nabla^{\ast
}(D_{i}\nabla_{jI}^{\alpha}-D_{j}\nabla_{iI}^{\alpha})=0,\quad|I|{}=k,
\]
For $f\in C^{\infty}(J^{k})$ and $I=i_{1}\cdots i_{s}$, put
\begin{equation}
\nabla_{I}f:=\nabla_{i_{1}}\cdots\nabla_{i_{s}}f\in C^{\infty}(J^{k}).
\label{1}%
\end{equation}
Definition (\ref{1}) is a good one since $\nabla$ is flat. In particular
$\nabla_{I}u^{\alpha}=u_{I}^{\alpha}$, for $|I|{}<k$.

Now, let $\nabla$ be a flat holonomic connection in $\pi_{k}$. Then for any
$\theta\in J^{k}$ there is a locally unique (local section) $s_{\theta}$ of
$\pi$ such that 1) $\theta=[s_{\theta}]_{x}^{k}$, $x=\pi_{k}(\theta)$ and 2)
$j_{k}s$ is a $\nabla$-constant section of $\pi_{k}$. Define a map
$\nabla_{\lbrack r]}:J^{k}\longrightarrow J^{k+r}$, $0\leq r\leq\infty$, by
putting
\[
\nabla_{\lbrack r]}(\theta):=[s_{\theta}]_{x}^{k+r},\quad x=\pi_{k}(\theta).
\]

\begin{proposition}
$\nabla_{\lbrack r]}$ is a smooth (closed) embedding locally given by
\[
\nabla_{\lbrack r]}^{\ast}(u_{I}^{\alpha})=\nabla_{I}u^{\alpha},\quad|I|{}\leq
k+r.
\]

\end{proposition}

\begin{proof}
$\nabla_{\lbrack r]}$ is clearly a section of the projection $\pi_{k+r,k}$.
Now, suppose that $\nabla_{\lbrack r]}^{\ast}(u_{J}^{\alpha})=\nabla
_{J}u^{\alpha}$ for some $|J|{}>k$, $|J|{}<k+r$, and let $\theta\in J^{k}$.
Put $x=\pi_{k}(\theta)$.Then
\begin{align*}
\nabla_{\lbrack r]}^{\ast}(u_{Ji}^{\alpha})(\theta)  &  =u_{Ji}^{\alpha
}([s_{\theta}]_{x}^{k+r})\\
&  =\frac{\partial^{|J|{}+1}s_{\theta}}{\partial x^{J}\partial x^{i}}(x)\\
&  =\left.  \frac{\partial}{\partial x^{i}}\right\vert _{x}\frac
{\partial^{|J|{}}s_{\theta}}{\partial x^{J}}\\
&  =\left.  \frac{\partial}{\partial x^{i}}\right\vert _{x}\nabla_{\lbrack
r]}^{\ast}(u_{J}^{\alpha})\circ j_{|J|}s_{\theta}\\
&  =(D_{i}\nabla_{J}u^{\alpha}\circ j_{|J|+1}s_{\theta})(x)\\
&  =(D_{i}\nabla_{J}u^{\alpha}\circ j_{k+1}s_{\theta})(x)\\
&  =(D_{i}\nabla_{J}u^{\alpha}\circ\nabla)(\theta)\\
&  =\nabla^{\ast}(D_{i}\nabla_{J}u^{\alpha})(\theta)\\
&  =(\nabla_{Ji}u^{\alpha})(\theta).
\end{align*}
By induction on $|J|$ the proposition follows.
\end{proof}

\begin{corollary}
The vector fields $\nabla_{i}$ and $D_{i}\in D(J^{\infty}\pi)$ are
$\nabla_{\lbrack\infty]}^{\ast}$-related, i.e., $\nabla_{\lbrack\infty]}%
^{\ast}\circ D_{i}=\nabla_{i}\circ\nabla_{\lbrack\infty]}^{\ast}$.
\end{corollary}

\begin{proof}
Compute
\[
(\nabla_{\lbrack\infty]}^{\ast}\circ D_{i})(x^{j})=\delta_{i}^{j}=(\nabla
_{i}\circ\nabla_{\lbrack\infty]}^{\ast})(x^{j})
\]
and
\begin{align*}
(\nabla_{\lbrack\infty]}^{\ast}\circ D_{i})(u_{I}^{\alpha})  &  =\nabla
_{\lbrack\infty]}^{\ast}(u_{Ii}^{\alpha})\\
&  =\nabla_{Ii}u^{\alpha}\\
&  =\nabla_{i}(\nabla_{I}u^{\alpha})\\
&  =(\nabla_{i}\circ\nabla_{\lbrack\infty]}^{\ast})(u_{I}^{\alpha}).
\end{align*}

\end{proof}

One can interpret $\mathscr{Y}^{\nabla}:=\operatorname{im}\nabla$ as a PDE on
sections of $\pi$. Then $\nabla$-constant sections are the $l$th jets of
solutions of $\mathscr{Y}^{\nabla}$.

\begin{corollary}
$\mathscr{Y}_{r}^{\nabla}=\operatorname{im}\nabla_{\lbrack r]}$ for all
$r\leq\infty$.
\end{corollary}

\begin{proof}
$\mathscr{Y}^{\nabla}$ is locally given by
\[
F_{I}^{\alpha}=0,\quad|I|{}\leq k,
\]
where $F_{I}^{\alpha}:=$ $u_{I}^{\alpha}-\nabla_{I}u^{\alpha}$. Therefore,
$\mathscr{Y}_{r}^{\nabla}$ is locally given by
\[
D_{J}F_{I}^{\alpha}=0,\quad|J|{}\leq r,\quad
\]
Now,
\begin{align*}
D_{J}F_{I}^{\alpha}  &  =D_{J}u_{I}^{\alpha}-D_{J}\nabla_{I}u^{\alpha}\\
&  =u_{IJ}^{\alpha}-(D_{J}\circ\nabla_{\lbrack r]}^{\ast})(u_{I}^{\alpha})\\
&  =u_{IJ}^{\alpha}-(\nabla_{\lbrack r]}^{\ast}\circ\nabla_{J})(u_{I}^{\alpha
}).
\end{align*}
But $u_{I}^{\alpha}=\nabla_{I}u^{\alpha}$ on $\mathscr{Y}_{r}^{\nabla}$. We
conclude that $\mathscr{Y}_{r}^{\nabla}$ is locally given by
\[
F_{IJ}^{\alpha}=0,\quad|I|{}\leq k,\ |J|{}\leq r.
\]

\end{proof}

Notice that $\operatorname{im}\nabla_{\lbrack\infty]}=\mathscr{Y}_{\infty}^{\nabla}$ is an
HJ subdiffiety of $J^{\infty}$. In particular, $\mathscr{Y}_{r}^{\nabla}$ is foliated by the graphs of the
$(k+r)$th jets of $\nabla$-constant sections.

\begin{remark}
$\nabla_{\lbrack\infty]}:(J^{k},\nabla)\longrightarrow(\mathscr{Y}_{\infty
}^{\nabla},\mathscr{C})$ is an isomorphism of bundles (over $M$) with (flat) connections.
\end{remark}

Now, let $\mathscr{E}\subset J^{\infty}$ be the infinite prolongation of a
$k$th-order differential equation $\mathscr{E}_{0}\subset J^{k}$ determined by
a differential operator $F:J^{k}\longrightarrow V$, $V\longrightarrow M$ being
a vector bundle, i.e., $F$ is a morphism of the bundles $\pi_{k}$ and
$V\longrightarrow M$ and
\[
\mathscr{E}_{0}=\{\theta\in J^{k}:F(\theta)=0\}.
\]

\begin{proposition}
Let $\nabla$ be a flat, holonomic connection in $\pi_{s}$, $s<k$. Then
$\mathscr{Y}_{\infty}^{\nabla}\subset\mathscr{E}$ iff
\begin{equation}
F\circ\nabla_{\lbrack k-s-1]}=0 \label{HJE}%
\end{equation}

\end{proposition}

\begin{proof}
The if implication is obvious. To prove the only if implication, notice that
if $\mathscr{Y}_{\infty}^{\nabla}\subset\mathscr{E}$, then all $\nabla
$-constant sections are $s$th jets of solutions of $\mathscr{E}$. Since
$\mathscr{Y}_{k-s}^{\nabla}$ is foliated by the graphs of $k$-jets of $\nabla
$-constant sections, the assertion follows.
\end{proof}

If $\mathscr{E}_{0}$ is locally given by
\[
F_{a}(\ldots,x^{i},\ldots,u_{I}^{\alpha},\ldots)=0,\quad|I|{}\leq k,
\]
then (\ref{HJE}) is locally given by
\begin{equation}
F_{a}(\ldots,x^{i},\ldots,\nabla_{I}u^{\alpha},\ldots)=0,\quad|I|{}\leq k,
\label{2}%
\end{equation}
which is a system of $(k-s-1)$th-order differential equation for the
$\nabla_{Ji}$'s $|J|{}=s$.

\begin{definition}
\label{Def1}Let $\nabla$ be a flat, holonomic connection in a jet bundle and
let $\mathscr{E}\subset J^{\infty}$ be a diffiety. The HJ diffiety
$\mathscr{Y}_{\infty}^{\nabla}$ will be called an \emph{elementary}
\emph{Hamilton-Jacobi (HJ) diffiety}. If\emph{ }$\mathscr{Y}_{\infty}^{\nabla
}\subset\mathscr{E}$, then $\mathscr{Y}_{\infty}^{\nabla}$ will be called an
\emph{elementary} \emph{HJ subdiffiety} of $\mathscr{E}$. Eq.{} (\ref{HJE})
for $\nabla$ will be called the ($s$th, \emph{generalized}) \emph{HJ equation}
of $\mathscr{E}$, $s=0,1,\ldots,k-1$.
\end{definition}

\begin{example}
\label{ExEL}Let $\pi:\mathbb{R}\times\mathbb{R}^{n}\ni(t,\boldsymbol{x})\longmapsto
t\in\mathbb{R}$ be the trivial bundle. Consider the Euler--Lagrange equations
(\ref{OEL}) in $J^{2}\pi$. Let
\[
\nabla=(dx^{i}-X^{i}dt)\otimes\dfrac{\partial}{\partial x^{i}}%
\]
be a connection in $\pi$, $X^{i}=X^{i}(t,x)$. Since $\pi$ is a bundle over a
$1$-dimensional manifold, $\nabla$ is automatically flat. Eq.{} (\ref{15})
is then the $0$th generalized HJ equation of Eq.{} (\ref{OEL}). This
motivates the choice of name for Eq.{} (\ref{HJE}).
\end{example}

\begin{remark}
Notice that the definition of elementary HJ subdiffiety is not intrinsic (see
Remark \ref{Remark4}) to a given diffiety $\mathscr{E}$, and it actually
depends on the embedding $\mathscr{E}\subset J^{\infty}$. Namely, it is easily
seen by dimensional arguments that changing the embedding $\mathscr{E}\subset
J^{\infty}$ could result in the transformation of an elementary HJ subdiffiety
into a new HJ subdiffiety which does not correspond to any holonomic
connection. This is why, urged by an anonymous referee, we gave
definition \ref{Def2} as the more fundamental one. However, we will mainly
consider the case when $\mathscr{E}$ emerges from a Lagrangian field theory as
the diffiety corresponding to the (Euler--Lagrange) field equations. In this
case $\mathscr{E}$ comes with a canonical embedding $\mathscr{E}\subset
J^{\infty}=J^{\infty}\pi$, where sections of $\pi$ are field configurations,
and the use of elementary HJ diffieties is very natural (see Sections
\ref{SecLHFT} and \ref{SecHJEL} for details).
\end{remark}

In the remaining part of this section we briefly discuss the relation between
general HJ diffieties and elementary ones.

\begin{proposition}
Let $\mathscr{O}\subset J^{\infty}$ be an HJ diffiety. Then, locally, there
exists $k$ such that $\mathscr{O}\subset\mathscr{Y}_{\infty}^{\nabla}$ for
some flat holonomic connection $\nabla$ in $\pi_{k}$.
\end{proposition}

\begin{proof}
Let $\mathscr{O}\subset J^{\infty}$ be an HJ diffiety. $\mathscr{O}=O_{\infty}$
for some differential equation $O\subset J^{k+1}$. Since $\mathscr{O}$ is
finite dimensional, $k$ can be chosen such that $\dim O=\dim O_{1}=\cdots
=\dim\mathscr{O}$. In particular, for any $\theta\in O$, $\mathscr{C}_{k}%
(\theta)\cap T_{\theta}O$ contains just one $R$-plane $\Theta$. Therefore, $O$
is $\pi_{k+1,k}$-horizontal (otherwise its tangent space would contain some
$\pi_{k+1,k}$-vertical tangent vector and one could find more $R$-planes in
$\mathscr{C}_{k}(\theta)\cap T_{\theta}O$). Put $N:=\pi_{k+1,k}(O)$. Locally
$N$ is a submanifold in $J^{k}$ such that $\dim N=\dim O$. If $\dim N=\dim
J^{k}$, then $O$ is already the image of a local, flat, holonomic connection
in $\pi_{k}$. Thus assume that $\dim N<\dim O$. Suppose that $\theta\in O$ and
$\theta^{\prime}=\pi_{k+1,k}(\theta)\in N$. In view of remark \ref{Remark3},
there exists at least one Lie field $X\in D(J^{k})$ such that $X_{\theta
^{\prime}}$ is transverse to $N$. Thus $X$ is transverse to $N$ locally around
$\theta^\prime$. Points in $O$ may be understood as $R$-planes at points of $N$
\cite{b...99}. They form an $n$-dimensional, involutive distribution on $N$.
Transporting both $N$ and the distribution on it along the flow of $X$, we may
produce a new submanifold $N^{\prime}\subset J^{k}$ with an involutive
distribution on it made of $R$-planes. It corresponds to a $\pi_{k+1,k}%
$-horizontal submanifold $O^{\prime}$ in $J^{k+1}$ locally containing $O$ (in
fact $O^{\prime}$ is obtained by transporting $O$ along the flow of $X_{1}\in
D(J^{k+1})$). If $\dim N^{\prime}=\dim J^{k}$, then $O^{\prime}$ is already
the image of a local, flat, holonomic connection in $\pi_{k}$. Otherwise we
may iterate the procedure. In the end we will produce the connection that we are
searching for.
\end{proof}

\begin{remark}
\label{Rem1}Let $\mathscr{E}\subset J^{\infty}$ be the infinite prolongation
of a $k$th-order differential equation $\mathscr{E}_{0}\subset J^{k}$ and let
$\mathscr{O}\subset\mathscr{E}$ be a finite dimensional subdiffiety. Let
$\mathscr{O}$ be the infinite prolongation of a submanifold $O\subset J^{s}$,
$s<k$, with $\dim O=\dim O_{1}=\cdots=\dim\mathscr{O}$. Clearly $O_{k-s}%
\subset\mathscr{E}_{0}$. Suppose that $r\leq k-s$, $\ell$ is the codimension of $O_{r}$
in $J^{s+r}$, and $\theta\in O_{r}$. If $\mathscr{E}_{0}$ possesses $\ell$ Lie
symmetries transverse to $O_{r}$ at $\theta$, then $\mathscr{O}$ can be
locally extended to $\mathscr{Y}_{\infty}^{\nabla}$ for some flat holonomic
connection $\nabla$ in $\pi_{r-1}$, such that $\mathscr{Y}_{\infty}^{\nabla
}\subset\mathscr{E}$. This can be easily proved along the same lines as in the
proof of the previous proposition.
\end{remark}

\section{Examples of HJ\ Diffieties\label{SecHJDiffEx}}

\begin{example}
Consider the Burgers equation
\[
u_{t}=u_{xx}+uu_{x}.
\]
It may be understood as a submanifold in $J^{2}\pi$ with $\pi$ the trivial
bundle $\pi:\mathbb{R}^{3}\ni(t,x,u)\longmapsto(t,x)\in\mathbb{R}$. Let
\[
\nabla=(du-Adt-Bdx)\otimes\frac{\partial}{\partial u}%
\]
be a connection in $\pi$, $A=A(t,x,u)$, $B=B(t,x,u)$. $\nabla$ is flat iff
\begin{equation}
B_{t}-A_{x}+AB_{u}-BA_{u}=0. \label{4}%
\end{equation}
The $0$th generalized HJ equation\ reads
\begin{equation}
A=B_{x}-BB_{u}-uB. \label{6}%
\end{equation}
Substituting (\ref{6}) into (\ref{4}) we find%
\begin{equation}
B_{t}-B_{xx}-2BB_{ux}-B^{2}B_{uu}-uB_{x}-B^{2}=0. \label{7}%
\end{equation}
Search for solutions of (\ref{7}) in the form $B=\alpha(x,t)u$. We must have
\begin{align*}
\alpha_{x}  &  =-\alpha^{2}\\
\alpha_{t}  &  =0
\end{align*}
and, therefore $\alpha=1/(x-x_{0})$, $x_{0}$ being an integration constant.
Thus $B=u/(x-x_{0})$ and $A=u^{2}/(x-x_{0})$. $\nabla$-constant sections are
the solutions of the (compatible) system
\[
\left\{
\begin{array}
[c]{l}%
u_{t}=u^{2}/(x-x_{0})\\
u_{x}=u/(x-x_{0})
\end{array}
\right.  ,
\]
i.e.,
\[
u=-\frac{x-x_{0}}{t-t_{0}}%
\]
($t_{0}$ being a new integration constant), which are indeed solutions of the
Burgers equation.
\end{example}

\begin{example}
In the same bundle as in the previous example, consider the heat equation
\[
u_{t}=u_{xx}.
\]
Let $\nabla$ be as above. The $0$th generalized HJ equation\ reads
\begin{equation}
A=B_{x}+BB_{u}. \label{3}%
\end{equation}
Substituting (\ref{3}) into (\ref{4}) we find
\begin{equation}
B_{t}-B_{xx}-2BB_{ux}-B^{2}B_{uu}=0. \label{5}%
\end{equation}
Search for solutions of (\ref{7}) in the form $B=\alpha(x,t)u$. We must have
\[
\alpha_{t}-\alpha_{xx}-2\alpha\alpha_{x}=0,
\]
One solution is $\alpha=\dfrac{1}{2}\phi$, $\phi$ being a solution of the
Burgers equation. Choose, for instance $\phi=-x/t$. Then
\begin{align*}
B  &  =-\dfrac{x}{2t}u,\\
A  &  =\dfrac{x^{2}-2t}{4t^{2}}u.
\end{align*}
$\nabla$-constant sections are the solutions of the (compatible) system
\[
\left\{
\begin{array}
[c]{l}%
u_{t}=\dfrac{x^{2}-2t}{4t^{2}}u\\
u_{x}=-\dfrac{x}{2t}u
\end{array}
\right.  ,
\]
i.e.,
\[
u=u_{0}\exp\left[  -\left(  \dfrac{1}{\sqrt{t}}+\dfrac{x^{2}}{4t}\right)
\right]
\]
($u_{0}$ being a new integration constant), which are indeed solutions of the
heat equation.
\end{example}

\begin{example}
\label{Examp1}In the same bundle as in the previous examples, consider the KdV
equation
\begin{equation}
u_{t}=6uu_{x}-u_{xxx}, \label{KdV}%
\end{equation}
and the corresponding diffiety $\mathscr{E}_{\mathrm{KdV}}$. The $2$nd-order
system of PDEs
\[
O:\left\{
\begin{array}
[c]{l}%
u_{t}=0\\
u_{xx}=3u^{2}\\
u_{xt}=0\\
u_{tt}=0
\end{array}
\right.
\]
is a $4$-dimensional one and determines a $4$-dimensional HJ subdiffiety
$\mathscr{O}$ of $\mathscr{O\subset E}_{\mathrm{KdV}}$. We now search for an
elementary HJ subdiffiety of $\mathscr{O\subset E}_{\mathrm{KdV}}$ containing
$\mathscr{O}$.

The Galilean boost
\begin{equation}
Y=-6t\frac{\partial}{\partial x}+\frac{\partial}{\partial u} \label{11}%
\end{equation}
is a Lie symmetry of (\ref{KdV}) and its second prolongation
\[
Y^{(2)}=Y+6u_{x}\frac{\partial}{\partial u_{t}}+6u_{xx}\frac{\partial
}{\partial u_{xt}}+12u_{xt}\frac{\partial}{\partial u_{tt}}%
\]
is transverse to $O$. Denote its flow by $\{A_{\tau}\}$. Then $\overline
{O}=\bigcup_{\tau}A_{\tau}(O)$ is a $5$-dimensional $2$nd-order system of
PDEs determining a $5$-dimensional subdiffiety $\overline{\mathscr{O}}%
\subset\mathscr{E}_{\mathrm{KdV}}$. Moreover, in view of Remark \ref{Rem1},
$\overline{\mathscr{O}}=\mathscr{Y}_{\infty}^{\nabla}$ for some flat holonomic
connection $\nabla$ in $\pi_{1}$. In particular, $\nabla$ is a solution of the
$1$st generalized HJ equation. Let us determine $\overline{O}$. $A_{\tau}$ is
given by
\begin{gather*}
A_{\tau}^{\ast}(x)=x-6\tau t,\quad A_{\tau}^{\ast}(t)=t,\quad A_{\tau}^{\ast
}(u)=u+\tau,\\
A_{\tau}^{\ast}(u_{x})=u_{x},\quad A_{\tau}^{\ast}(u_{t})=u_{t}+6\tau u_{x},\\
A_{\tau}^{\ast}(u_{xx})=u_{xx},\quad A_{\tau}^{\ast}(u_{xt})=u_{xt}+6\tau
u_{xx},\quad A_{\tau}^{\ast}(u_{tt})=u_{tt}+12\tau u_{xt}+36\tau^{2}u_{xx}.
\end{gather*}
Therefore, $\overline{O}$ is parametrically given by
\[
\overline{O}:\left\{  {%
\begin{array}
[c]{l}%
x=y-6\tau s\\
t=s\\
u=v+\tau\\
u_{x}=p\\
u_{t}=6\tau p\\
u_{xx}=3v^{2}\\
u_{xt}=18\tau v^{2}\\
u_{tt}=108\tau^{2}v^{2}%
\end{array}
}\right.  .
\]
Eliminating the $5$ parameters $(y,s,v,p,\tau)$ we get
\[
\overline{O}:\left\{  {%
\begin{array}
[c]{l}%
u_{xx}=\dfrac{(u_{t}-6uu_{x})^{2}}{12u_{x}^{2}}\\
u_{xt}=\dfrac{u_{t}(u_{t}-6uu_{x})^{2}}{12u_{x}^{3}}\\
u_{tt}=\dfrac{u_{t}^{2}(u_{t}-6uu_{x})^{2}}{12u_{x}^{4}}%
\end{array}
}\right.  ,
\]
and $\nabla$ is given by
\[
\nabla=(du-u_{x}dx-u_{t}dt)\otimes\frac{\partial}{\partial u}+(du_{x}%
-Adx-Cdt)\otimes\frac{\partial}{\partial u_{x}}+(du_{t}-Cdx-Bdt)\otimes
\frac{\partial}{\partial u_{t}}%
\]
with
\[
A=\dfrac{(u_{t}-6uu_{x})^{2}}{12u_{x}^{2}},\quad B=-\dfrac{u_{t}^{2}%
(u_{t}-6uu_{x})^{2}}{12u_{x}^{4}},\quad C=\dfrac{u_{t}(u_{t}-6uu_{x})^{2}%
}{12u_{x}^{3}}.
\]
A direct computation shows that $\nabla$ is indeed flat and it is a solution
of the $1$st generalized HJ equation
\[
A_{x}+u_{x}A_{u}+AA_{u_{x}}+CA_{u_{t}}+u_{t}-6uu_{x}=0.
\]
Clearly, solutions of $\overline{O}$, or, which is the same, $\nabla$-constant
sections, are boosted solutions of $O$. Namely, solutions of $O$ are
\begin{equation}
u=2^{1/3}\wp(2^{-1/3}(x-c_{0});0,c_{1}), \label{12}%
\end{equation}
$\wp(z;\omega_{1},\omega_{2})$ being the Weierstrass elliptic function of $z$,
with periods $\omega_{1},\omega_{2}$, and $c_{0},c_{1}$ integration constants.
Solutions of $\overline{O}$ are then found by transporting (\ref{12}) along the
flow of $Y$. They are
\[
u=2^{1/3}\wp(2^{-1/3}(x-ct-c_{0});0,c_{1})+c/6,
\]
where $c$ is a new integration constant, and they are (local) solutions of the
KdV equation.
\end{example}

\section{The Lagrangian-Hamiltonian Formalism in Field Theory\label{SecLHFT}}

In this section we review the Lagrangian-Hamiltonian formalism for higher
derivative field theories. Details can be found in \cite{v10b} (see also
\cite{v10}).

\begin{definition}
A \emph{Lagrangian field theory }of order $\leq k+1$, $0\leq k<\infty$, is a
pair $(\pi,\mathscr{L})$, where $\pi:E\longrightarrow M$ is a fiber bundle and
$\mathscr{L}\in\Lambda_{n}^{n}(J^{k+1},\pi_{k+1})\subset\overline{\Lambda}%
{}^{n}$, $n=\dim M$, is a Lagrangian density.
\end{definition}

As already recalled, the horizontal cohomology class $[\mathscr{L}]\in
H^{n}(\overline{\Lambda},\overline{d})$ identifies with the action functional
$\int_{M}\mathscr{L}$ which is extremized by solutions of the Euler--Lagrange
(EL) field equations
\[
\boldsymbol{E}(\mathscr{L})=0.
\]
The EL equations are $2(k+1)$th-order PDEs. Denote by $\mathscr{E}_{EL}\subset
J^{\infty}$ the corresponding diffiety.

The Lagrangian field theory $(\pi,\mathscr{L})$ determines a canonical PD
Hamiltonian system $\omega_{\mathscr{L}}\in\Lambda^{n+1}(J^{\dag}\pi_{\infty
})$ in $J^{\dag}\pi_{\infty}\longrightarrow M$ (the reduced multimomentum
bundle of $\pi_{\infty}$). If $\mathscr{L}$ is locally given by
$\mathscr{L}=Ld^{n}x$, $L\in C^{\infty}(J^{k+1})$, then $\omega_{\mathscr{L}}$
is locally given by
\[
\omega_{\mathscr{L}}=dp_{\alpha}^{I.i}\wedge du_{I}^{\alpha}\wedge
d^{n-1}x_{i}-dE_\mathscr{L}\wedge d^{n}x,\quad E_{\mathscr{L}}=u_{Ii}^{\alpha}p_{\alpha
}^{I.i}-L,
\]
the $p_{\alpha}^{I.i}$'s being momentum coordinates associated with the
$u_{I}^{\alpha}$. The corresponding PD Hamilton equations
\begin{equation}
i^{1,n}(j_{1}\sigma)\omega_{\mathscr{L}}|_{\sigma}=0, \label{ELHE}%
\end{equation}
for sections $\sigma$ of $J^{\dag}\pi_{\infty}\longrightarrow M$, locally
read
\[
\left\{
\begin{array}
[c]{l}%
p_{\alpha}^{I.i},_{i}=\dfrac{\partial L}{\partial u_{I}^{\alpha}}-\delta
_{Ji}^{I}\,p_{\alpha}^{J.i}\\
u_{I}^{\alpha},_{i}=u_{Ii}^{\alpha}%
\end{array}
\right.
\]
where $({}\bullet{}),_{i}$ denotes differentiation of $({}\bullet{})$ with
respect to $x^{i}$, $i=1,\ldots,n$, and $\delta_{K}^{I}=1$ when the multi-indices
$I,K$ coincide, and $\delta_{K}^{I}=0$ otherwise. We will refer to equations
(\ref{ELHE}) as the \emph{Euler--Lagrange--Hamilton (ELH) equations on }%
$J^{\dag}\pi_{\infty}$. They can be interpreted as Euler--Lagrange equations
corresponding to an Hamilton--Pontryagin-like variational principle \cite{l19}
and their solutions are characterized by the following

\begin{theorem}
\label{Theorem1}A section $\sigma$ of $J^{\dag}\pi_{\infty}\longrightarrow M$
is a solution of Eq.{} (\ref{ELHE}) iff, locally, $\sigma=\vartheta\circ
j_{\infty}s$, where $s$ is a solution of the EL equations and $\vartheta$ is a
Legendre form.
\end{theorem}

In particular, Equations (\ref{ELHE}) covers the EL equations, i.e., if $\sigma$ is
a solution of (\ref{ELHEk}), then the composition
\[
M\overset{\sigma}{\longrightarrow}J^{\dag}\pi_{\infty}\longrightarrow
J^{\infty}%
\]
is of the form $j_{\infty}s$ for a solution $s$ of the EL equations, and all
solutions of the EL equations can be obtained like this.

\begin{lemma}
\label{Lemma1}Let $T$ be a section of $J^{\dag}\pi_{\infty}\longrightarrow
J^{\infty}$, i.e., $T\in\mathscr{C}\Lambda^{1}\otimes\overline{\Lambda}{}^{n-1}$, then
\[
T^{\ast}\omega_{\mathscr{L}}=d(\mathscr{L}+T).
\]

\end{lemma}

\begin{proof}
The lemma can be proved in a coordinate free manner, using the intrinsic
definition of $\omega_{\mathscr{L}}$. We here propose a local proof. Let $T$
be locally given by $T^{\ast}(p_{\alpha}^{I.i})=T_{\alpha}^{I.i}\in C^{\infty
}(J^{\infty})$, i.e.,
\[
T=T_{\alpha}^{I.i}d^{V}\!u_{I}^{\alpha}\otimes d^{n-1}x_{i}.
\]
Notice that $\omega_{\mathscr{L}}%
=d\varrho_{\mathscr{L}}$, with $\varrho_{\mathscr{L}}\in\Lambda^{n}(J^{\dag
}\pi_{\infty})$ locally given by
\[
\varrho_{\mathscr{L}}=p_{\alpha}^{I.i}du_{I}^{\alpha}\wedge d^{n-1}%
x_{i}-Ed^{n}x.
\]
Then
\begin{align*}
T^{\ast}\omega_{\mathscr{L}}  &  =dT^{\ast}\varrho_{\mathscr{L}}\\
&  =d[T_{\alpha}^{I.i}du_{I}^{\alpha}\wedge d^{n-1}x_{i}-(u_{Ii}^{\alpha
}T_{\alpha}^{I.i}-L)d^{n}x]\\
&  =d(T_{\alpha}^{I.i}d^{V}\!u_{I}^{\alpha}\wedge d^{n-1}x_{i}+Ld^{n}x)\\
&  =d(T+\mathscr{L}).
\end{align*}

\end{proof}

The Lagrangian field theory $(\pi,\mathscr{L})$ determines two more canonical
PD Hamiltonian systems. First of all, the PD Hamiltonian system $\omega
_{\mathscr{L}}$ is the pull-back of a unique PD Hamiltonian system $\omega$ in
$\pi_{k+1,k}^{\circ}(J^{\dag}\pi_{k})\longrightarrow M$. $\omega$ is locally
given by%
\[
\omega=\sum_{|I|{}\leq k}dp_{\alpha}^{I.i}\wedge du_{I}^{\alpha}d^{n-1}%
x_{i}-dE\wedge d^{n}x,\quad E=\sum_{|I|{}\leq k}u_{Ii}^{\alpha}p_{\alpha
}^{I.i}-L.
\]
The associated PD Hamilton equations
\begin{equation}
i^{1,n}(j_{1}\sigma)\omega|_{\sigma}=0, \label{ELHEk}%
\end{equation}
for sections $\sigma$ of $\pi_{k+1,k}^{\circ}(J^{\dag}\pi_{k})\longrightarrow
M$ locally read
\begin{equation}
\left\{
\begin{array}
[c]{llc}%
\dfrac{\partial L}{\partial u_{K}^{\alpha}}-\delta_{Ji}^{K}\,p_{\alpha}%
^{J.i}=0 & |K|{}=k+1 & \quad I\\
p_{\alpha}^{I.i},_{i}=\dfrac{\partial L}{\partial u_{I}^{\alpha}}-\delta
_{Ji}^{I}\,p_{\alpha}^{J.i} & |I|{}\leq k & \quad II\\
u_{J}^{\alpha},_{i}=u_{Ji}^{\alpha} & |J|{}\leq k & \quad III
\end{array}
\right.  . \label{22}%
\end{equation}
We will refer to them as the \emph{ELH equations on }$\pi_{k+1,k}^{\circ
}(J^{\dag}\pi_{k})$. They cover the EL equations as well, i.e., if $\sigma$ is
a solution of (\ref{ELHEk}), then the composition
\[
M\overset{\sigma}{\longrightarrow}\pi_{k+1,k}^{\circ}(J^{\dag}\pi
_{k})\longrightarrow J^{k}%
\]
is of the form $j_{k}s$ for a solution $s$ of the EL equations, and all
solutions of the EL equations can be obtained like this. (\ref{22}) shows that
solutions of (\ref{ELHEk}) take values in the \emph{first constraint subbundle
}$\mathscr{P}\longrightarrow M$, $\mathscr{P}\subset\pi_{k+1,k}^{\circ
}(J^{\dag}\pi_{k})$, which is locally defined by
\[
\dfrac{\partial L}{\partial u_{K}^{\alpha}}-\delta_{Ji}^{K}\,p_{\alpha}%
^{J.i}=0,\quad|K|{}=k+1.
\]

Let $\mathscr{P}_{0}\subset J^{\dag}\pi_{k}$ be the image of $\mathscr{P}$
under the projection $\pi_{k+1,k}^{\circ}(J^{\dag}\pi_{k})\longrightarrow
J^{\dag}\pi_{k}$. Under the (not too restrictive) hypothesis that the projection
$\mathscr{P}\longrightarrow\mathscr{P}_{0}$ is a smooth submersion with
connected fibers, $i_{\mathscr{P}}^{\ast}\omega$ is the pull-back of a unique
PD Hamiltonian system $\omega_{0}$ in the bundle $\mathscr{P}_{0}%
\longrightarrow M$. We will refer to the associated PD Hamilton equations
\begin{equation}
i^{1,n}(j_{1}\sigma_{0})\omega_{0}|_{\sigma_{0}}=0 \label{HDWE}%
\end{equation}
for sections $\sigma_{0}$ of $\mathscr{P}_{0}\longrightarrow M$ as the
\emph{Hamilton--de Donder--Weyl (HDW) equations}. When $(\pi,\mathscr{L})$ is a
(hyper)regular theory, i.e., $\mathscr{P}_{0}=J^{\dag}\pi_{k}$, they locally
coincides with the de Donder (higher derivative, Hamilton-like) field
equations \cite{d35}
\[
\left\{
\begin{array}
[c]{l}%
p_{\alpha}^{I.i},_{i}=-\dfrac{\partial H}{\partial u_{I}^{\alpha}}\\
u_{I}^{\alpha},_{i}=\dfrac{\partial H}{\partial p_{\alpha}^{I.i}}%
\end{array}
\right.  ,
\]
where the local function $H\in C^{\infty}(J^{\dag}\pi_{k})$ is uniquely
defined by the condition that its pull-back via the projection\emph{
}$\mathscr{P}\longrightarrow J^{\dag}\pi_{k}$ is $i_{\mathscr{P}}^{\ast}(E)\in
C^{\infty}(\mathscr{P})$.

If $\sigma$ is a solution of Eq.{} (\ref{ELHEk}), then the composition
\[
M\overset{\sigma}{\longrightarrow}\pi_{k+1,k}^{\circ}(J^{\dag}\pi
_{k})\longrightarrow J^{\dag}\pi_{k}%
\]
is a solution of (\ref{HDWE}). However, Eq.{} (\ref{HDWE}) generically
possesses more solutions than are the ones obtained like this, unless
$(\pi,\mathscr{L})$ is a (hyper)regular theory. In this case, the ELH
equations cover the HDW equations.

In \cite{v10}, we presented an higher derivative, field theoretic version of
the geometric HJ formalism \cite{c...06,c...09} in the case of an hyperregular
Lagrangian field theory. In the next section we generalize the constructions
and results of \cite{v10} to possibly singular Lagrangian field theories. The
concept of HJ diffiety will play a special role.

\section{Hamilton-Jacobi Subdiffieties of Euler--Lagrange
Equations\label{SecHJEL}}

In \cite{c...06,c...09,c...10}, Cari\~{n}ena, Gr\`{a}cia, Marmo,
Mart\'{\i}nez, Mu\~{n}oz-Lecanda, and Rom\'{a}n-Roy presented a geometric
formulation of the classical HJ theory of mechanical systems in both
Lagrangian and Hamiltonian settings. They also presented a generalized HJ
problem depending on the sole equations of motion, and not on the Lagrangian,
nor the Hamiltonian. Their formulation \textquotedblleft is based on the idea
of obtaining solutions of a second order differential equations by lifting
solutions of an adequate first order differential equation \cite{c...09}%
\textquotedblright. This idea can be generalized to the higher derivative,
regular, Lagrangian field theoretic setting \cite{v10}. We here propose a
further generalization to the (possibly) singular case.

Let $(\pi,\mathscr{L})$ be a $k$th-order Lagrangian field theory, and
$\omega\in\Lambda^{n+1}(\pi_{k+1,k}^{\circ}(J^{\dag}\pi_{k}))$ the associated
PD Hamiltonian system in $\pi_{k+1,k}^{\circ}(J^{\dag}\pi_{k})\longrightarrow
M$. A section of the pull-back bundle $\pi_{k+1,k}^{\circ}(J^{\dag}\pi
_{k})\longrightarrow J^{k}$ can be understood as a pair $(\nabla,T)$, where
$\nabla$ is an holonomic connection in $\pi_{k}$, and $T$ is a section of
$J^{\dag}\pi_{k}\longrightarrow J^{k}$, i.e., $T\in V\!\Lambda^{1}(J^{k},\pi
_{k})\otimes\Lambda_{n-1}^{n-1}(J^{k},\pi_{k})$. We will always adopt this
point of view. Obviously, the diagram
\[%
\xymatrix{ \pi_{k+1,k}^\circ(J^\dag\pi_k)  \ar[d] \ar[r] & J^\dag\pi_k  \\
J^{k+1}  & J^k \ar[lu]_-{(\nabla, T)} \ar[l]^-{\nabla} \ar[u]_-T
}%
\]
commutes.

\begin{lemma}
\label{Lemma2}Let $(\nabla,T)$ be a section of $\pi_{k+1,k}^{\circ}(J^{\dag
}\pi_{k})\longrightarrow J^{k}$, with $\nabla$ a flat (holonomic) connection.
Then, one has
\[
(\nabla,T)^{\ast}\omega=d(\nabla^{\ast}\mathscr{L}+e^{1,n-1}(\nabla)T).
\]

\end{lemma}

\begin{proof}
We can take the pull-back $\pi_{\infty,k}^{\ast}(T)\in\mathscr{C}\Lambda
^{1}\otimes\overline{\Lambda}{}^{n-1}$. In the following, abusing the
notation, we will indicate this pull-back by $T$ again. Denote by
$\mathfrak{p}:J^{\dag}\pi_{\infty}\longrightarrow\pi_{k+1,k}^{\circ}(J^{\dag
}\pi_{k})$ the canonical projection. Then 1) $(\nabla,T)=\mathfrak{p}\circ
T\circ\nabla_{\lbrack\infty]}$ (where, in the lhs we interpreted $T$ as a
section of $J^{\dag}\pi_{\infty}$), and 2) $\mathfrak{p}^{\ast}\omega
=\omega_{\mathscr{L}}$. Therefore,
\begin{align*}
(\nabla,T)^{\ast}\omega &  =\nabla_{\lbrack\infty]}^{\ast}T^{\ast}%
\mathfrak{p}^{\ast}\omega\\
&  =\nabla_{\lbrack\infty]}^{\ast}T^{\ast}\omega_{\mathscr{L}}\\
&  =\nabla_{\lbrack\infty]}^{\ast}d(\mathscr{L}+T)\\
&  =d\nabla_{\lbrack\infty]}^{\ast}(\mathscr{L}+T)\\
&  =d(\nabla^{\ast}\mathscr{L}+e^{1,n-1}(\nabla)T),
\end{align*}
where we used Lemma \ref{Lemma1} and, in the last line, the (obvious) fact
that $\nabla_{\lbrack\infty]}^{\ast}$ is a morphism of the variational
bicomplex and the bicomplex defined by $\nabla$.
\end{proof}

\begin{definition}
The \emph{generalized HJ problem} for the Lagrangian theory $(\pi
,\mathscr{L})$ consists in finding an holonomic flat connection $\nabla$ in
$\pi_{k}$ and a section $T$ of $J^{\dag}\pi_{k}\longrightarrow J^{k}$ such
that for every $\nabla$-constant section $\sigma$, $(\nabla,T)\circ\sigma$ is
a solution of the ELH equations.
\end{definition}

The relevance of the generalized HJ problem resides in the following

\begin{theorem}
\label{Prop1}Let $(\nabla,T)$ be a section of $\pi_{k+1,k}^{\circ}(J^{\dag}%
\pi_{k})\longrightarrow J^{k}$, with $\nabla$ a flat (holonomic) connection.
The following conditions are equivalent:

\begin{enumerate}
\item $(\nabla,T)$ is a solution of the generalized HJ problem;

\item $\operatorname{im}(\nabla,T)\subset\mathscr{P}$ and, for every $\nabla
$-constant section $j$, $T\circ j$ is a solution of the HDW equations;

\item $\operatorname{im}(\nabla,T)\subset\mathscr{P}$ (hence
$\operatorname{im}T\subset\mathscr{P}_{0}$) and $i^{1,n}(\nabla)T^{\ast}%
\omega_{0}=0$;

\item $\operatorname{im}(\nabla,T)\subset\mathscr{P}$ and $i^{1,n}%
(\nabla)(\nabla,T)^{\ast}\omega=0$.
\end{enumerate}

Moreover, each of the above conditions implies:

\begin{enumerate}
\item[5.] $\mathscr{Y}_{\infty}^{\nabla}$ is an (elementary) HJ subdiffiety of
$\mathscr{E}_{EL}$.
\end{enumerate}
\end{theorem}

\begin{proof}
1. $\Longrightarrow$ 2. Let $j$ be a $\nabla$-constant section. Then
$(\nabla,T)\circ j$ is a solution of the ELH equations and, therefore, takes
values in $\mathscr{P}$. Since $\nabla$-constant sections \textquotedblleft
foliate\textquotedblright\ $J^{k}$ we conclude that $(\nabla,T)$ itself takes
values in $\mathscr{P}$. Finally, the projection $\mathscr{P}\longrightarrow
\mathscr{P}_{0}$ maps solutions of the ELH equations to solutions of the HDW equations.

2. $\Longrightarrow$ 3. Let $j$ be a $\nabla$-constant section (and hence $T\circ
j$ is a solution of the HDW equations) and $X$ a $\pi_{k}$-vertical vector
field on $J^{k}$ along $j$. Compute
\[
i_{X}i^{1,n}(j_{1}j)T^{\ast}\omega_{0}|_{j}=i_{(dT)(X)}i^{1,n}(j_{1}(T\circ
j))\omega_{0}|_{T\circ j}=0.
\]
Since $X$ is arbitrary, $i^{1,n}(j_{1}j)T^{\ast}\omega_{0}|_{j}=0$. Moreover,
$\nabla$-constant sections \textquotedblleft foliate\textquotedblright%
\ $J^{k}$ and, therefore, $i^{1,n}(\nabla)T^{\ast}\omega_{0}=0$.

3. $\Longrightarrow$ 4. Suppose that $\operatorname{im}(\nabla,T)\subset\mathscr{P}$.
Then
\[
(\nabla,T)^{\ast}\omega=(\nabla,T)^{\ast}i_{\mathscr{P}}^{\ast}\omega=T^{\ast
}\omega_{0}.
\]

4. $\Longrightarrow$ 1. Suppose that $\operatorname{im}(\nabla,T)\subset\mathscr{P}$, with
$i^{1,n}(\nabla)(\nabla,T)^{\ast}\omega=0$, and let $j$ be a $\nabla$-constant
section. We prove that $(\nabla,T)\circ j$ is a solution of the ELH equations.
Indeed, in view of Lemma \ref{Lemma2},
\begin{align*}
0  &  =i^{1,n}(\nabla)(\nabla,T)^{\ast}\omega\\
&  =i^{1,n}(\nabla)d(\nabla^{\ast}\mathscr{L}+e^{1,n-1}(\nabla)T)\\
&  =d^{V}\!\nabla^{\ast}\mathscr{L}+\overline{d}_{\nabla}T.
\end{align*}
Now, locally,
\[
d^{V}\!\nabla^{\ast}\mathscr{L}+\overline{d}_{\nabla}T=\sum\nolimits_{|J|,|I|{}%
\leq k}\nabla^{\ast}\left(  \dfrac{\partial L}{\partial u_{I}^{\alpha}}%
-D_{i}T_{\alpha}^{I.i}-\delta_{Ji}^{I}T_{\alpha}^{J.i}\right)  d^{V}
\!u_{I}^{\alpha}\otimes d^{n}x.
\]
Thus,
\begin{equation}
\nabla^{\ast}\left(  \dfrac{\partial L}{\partial u_{I}^{\alpha}}%
-D_{i}T_{\alpha}^{I.i}-\delta_{Ji}^{I}T_{\alpha}^{J.i}\right)  =0. \label{8}%
\end{equation}
$\sigma:=(\nabla,T)\circ j$ satisfies Equations (\ref{22}) $I$ (because it
takes values in $\mathscr{P}$) and (\ref{22}) $III$ (because $j$ is $\nabla
$-constant). We show that it satisfies Eq.{} (\ref{22}) $II$ also. For
$|I|{}\leq k$
\begin{align*}
\sigma^{\ast}(p_{\alpha}^{I.i}),_{i}  &  =j^{\ast}(T_{\alpha}^{I.i})_{,i}\\
&  =(j_{1}j)^{\ast}D_{i}T_{\alpha}^{I.i}\\
&  =(\nabla\circ j)^{\ast}D_{i}T_{\alpha}^{I.i}\\
&  =j^{\ast}\nabla^{\ast}(D_{i}T_{\alpha}^{I.i})\\
&  =j^{\ast}\nabla^{\ast}\left(  \dfrac{\partial L}{\partial u_{I}^{\alpha}%
}-\delta_{Ji}^{I}T_{\alpha}^{J.i}\right) \\
&  =j^{\ast}(\nabla,T)^{\ast}\left(  \dfrac{\partial L}{\partial u_{I}%
^{\alpha}}-\delta_{Ji}^{I}p_{\alpha}^{J.i}\right) \\
&  =\sigma^{\ast}\left(  \dfrac{\partial L}{\partial u_{I}^{\alpha}}%
-\delta_{Ji}^{I}p_{\alpha}^{J.i}\right)  .
\end{align*}

1. $\Longrightarrow$ 5. Obvious, since the projection
$\mathscr{P}\longrightarrow E$ maps solutions of the ELH equations to
solutions of the EL equations.
\end{proof}

In view of the above theorem, given a solution $(\nabla,T)$ of the generalized
HJ problem we can obtain solutions of the ($2k$th-order) EL equations, finding
solutions of the much simpler ($k$th-order) equation $\mathscr{Y}^{\nabla}$.

We now prove that the last implication in the above proof can be inverted as
well in the following sense. If $\mathscr{Y}_{\infty}^{\nabla}$ is an HJ
subdiffiety of the EL equations then there exists $T$ such that $(\nabla,T)$
is a solution of the generalized HJ problem. This result is obtained by observing
the relation between the generalized HJ problem and Legendre forms.

\begin{theorem}
\label{13}Let $(\nabla,T)$ be a section of $\pi_{k+1,k}^{\circ}(J^{\dag}%
\pi_{k})\longrightarrow J^{k}$, with $\nabla$ a flat (holonomic) connection.
The following conditions are equivalent

\begin{enumerate}
\item $(\nabla,T)$ is a solution of the generalized HJ problem;

\item $(d^{V}\!\mathscr{L}+\overline{d}T)|_{\nabla_{\lbrack\infty]}}=0$;

\item \label{13,3}$\mathscr{Y}_{\infty}^{\nabla}$ is an (elementary) HJ
subdiffiety of $\mathscr{E}_{EL}$, and there exists a Legendre form
$\vartheta$ such that $(T-\vartheta)|_{\nabla_{\lbrack\infty]}}=0$.
\end{enumerate}
\end{theorem}

\begin{proof}
1. $\Longrightarrow$ 2. Recall, preliminarily, that, in view of Lemma
\ref{Lemma1}, $T^{\ast}\omega_{\mathscr{L}}=d(\mathscr{L}+T)=d^{V}\!%
\mathscr{L}+\overline{d}T+d^{V}\!T$. Now let $j$ be a $\nabla$-constant section.
Then $(\nabla,T)\circ j$ is a solution of the ELH equations on $\pi
_{k+1,k}^{\circ}(J^{\dag}\pi_{k})$. Coordinate formulas then show that
$T\circ\nabla_{\lbrack\infty]}\circ j$ is a solution of the ELH equations on
$J^{\dag}\pi_{\infty}$, i.e.,
\[
i^{1,n}(j_{1}(T\circ\nabla_{\lbrack\infty]}\circ j))\omega_{\mathscr{L}}%
|_{T\circ\nabla_{\lbrack\infty]}\circ j}=0.
\]
Let $X$ be a $\pi_{\infty}$-vertical vector field over $J^{\infty}$ along
$\nabla_{\lbrack\infty]}\circ j$. Then
\[
i_{X}i^{1,n}(j_{1}(\nabla_{\lbrack\infty]}\circ j))T^{\ast}\omega
_{\mathscr{L}}|_{\nabla_{\lbrack\infty]}\circ j}=i_{(dT)(X)}i^{1,n}%
(j_{1}(T\circ\nabla_{\lbrack\infty]}\circ j))\omega_{\mathscr{L}}%
|_{T\circ\nabla_{\lbrack\infty]}\circ j}=0.
\]
Since $X$ is arbitrary
\begin{align*}
0  &  =i^{1,n}(j_{1}(\nabla_{\lbrack\infty]}\circ j))T^{\ast}\omega
_{\mathscr{L}}|_{\nabla_{\lbrack\infty]}\circ j}\\
&  =i^{1,n}(j_{1}(\nabla_{\lbrack\infty]}\circ j))(d^{V}\!\mathscr{L}+\overline
{d}T+d^{V}\!T)|_{\nabla_{\lbrack\infty]}\circ j}\\
&  =[i^{1,n}(\mathscr{C})(d^{V}\!\mathscr{L}+\overline{d}T+d^{V}\!T)]|_{\nabla
_{\lbrack\infty]}\circ j}\\
&  =(d^{V}\!\mathscr{L}+\overline{d}T)|_{\nabla_{\lbrack\infty]}\circ j}.
\end{align*}
Since $\nabla$-constant sections foliate $J^{k}$, we get $(d^{V}\!%
\mathscr{L}+\overline{d}T)|_{\nabla_{\lbrack\infty]}}=0$.

2. $\Longrightarrow$ 3. $0=(d^{V}\!\mathscr{L}+\overline{d}T)|_{\nabla
_{\lbrack\infty]}}=(\boldsymbol{E}(\mathscr{L})+\overline{d}(T-\vartheta
_{0}))|_{\nabla_{\lbrack\infty]}}$, where $\vartheta_{0}$ is a Legendre form.
Therefore $\boldsymbol{E}(\mathscr{L})|_{\nabla_{\lbrack\infty]}}=\overline
{d}(\vartheta_{0}-T)|_{\nabla_{\lbrack\infty]}}=\overline{d}|_{\nabla_{\infty
}}(\vartheta_{0}-T)|_{\nabla_{\infty}}$, where we used that
$\mathscr{Y}_{\infty}^{\nabla}\subset J^{\infty}$ is a subdiffiety. Recall
that, in view of Remark \ref{Remark1}, $\boldsymbol{E}(\mathscr{L})|_{\nabla
_{\lbrack\infty]}}$ cannot be $\overline{d}|_{\nabla_{\lbrack\infty]}}$-exact
unless it is $0$. We conclude that $\mathscr{Y}_{\infty}^{\nabla}%
\subset\mathscr{E}_{EL}$. Moreover, $(\vartheta_{0}-T)|_{\nabla_{\lbrack
\infty]}}$ is $\overline{d}|_{\nabla_{\lbrack\infty]}}$-closed, and hence
$\overline{d}|_{\nabla_{\lbrack\infty]}}$-exact, i.e., there exists $\nu
\in\mathscr{C}\Lambda^{1}\otimes\overline{\Lambda}{}^{n-2}$ such that
$(\vartheta_{0}-T)|_{\nabla_{\lbrack\infty]}}=\overline{d}|_{\nabla
_{\lbrack\infty]}}\nu|_{\nabla_{\lbrack\infty]}}$ or, which is the same,
$(T-\vartheta)|_{\nabla_{\lbrack\infty]}}=0$ where we put $\vartheta
=\vartheta_{0}-\overline{d}\nu$. Finally, notice that $\vartheta$ itself is a
Legendre form.

3. $\Longrightarrow$ 1. Let $j$ be a $\nabla$-constant section. Then
$j=j_{k}s$ for some solution of the EL equations and $T\circ j_{\infty
}s=\vartheta\circ j_{\infty}s$. In view of Theorem \ref{Theorem1}, $T\circ
j_{\infty}s$ is a solution of the ELH equations on $J^{\dag}\pi_{\infty}$. We
conclude that
\[
\pi_{k+1,k}^{\ast}(T)\circ j_{k+1}s=\pi_{k+1,k}^{\ast}(T)\circ\nabla\circ
j_{k}s=(\nabla,T)\circ j_{k}s=(\nabla,T)\circ j
\]
is a solution of the ELH equations on $\pi_{k+1,k}^{\circ}(J^{\dag}\pi_{k})$.
\end{proof}

\begin{corollary}
Let $\nabla$ be an holonomic, flat connection in $\pi_{k}$. There exists $T$
such that $(\nabla,T)$ is a solution of the generalized HJ problem iff
$\mathscr{Y}_{\infty}^{\nabla}$ is an (elementary) HJ subdiffiety of
$\mathscr{E}_{EL}$.
\end{corollary}

\begin{proof}
The if implication is already stated in Theorem \ref{Prop1}, point 5.
Conversely, if $\mathscr{Y}_{\infty}^{\nabla}\subset\mathscr{E}_{EL}$, then
$\boldsymbol{E}(\mathscr{L})|_{\mathscr{\nabla}_{[\infty]}}=0$. Let
$\vartheta$ be a Legendre form depending only on (vertical differentials of)
derivatives up to the order $k$. Put $T:=\vartheta|_{\nabla_{\lbrack\infty]}}%
$. $T$ is a section of $J^{\dag}\pi_{k}\longrightarrow J^{k}$. Moreover,
\[
(d^{V}\!\mathscr{L}+\overline{d}T)|_{\nabla_{\lbrack\infty]}}=(\boldsymbol{E}%
(\mathscr{L})-\overline{d}\vartheta+\overline{d}T)|_{\mathscr{\nabla}_{[\infty
]}}=0.
\]
Now, use Theorem \ref{13}.
\end{proof}

The above corollary shows that equation $i^{1,n}(\nabla)(\nabla,T)^{\ast
}(\omega)=0$ covers the generalized ($k$th-order) HJ equation of the EL
equations. Since all the Legendre forms of a given Lagrangian density are
known, it follows that solving the generalized HJ problem is basically
equivalent to finding all ($k$th-order) HJ subdiffieties of the EL equations.

\begin{corollary}
Let $\mathscr{L}$ and $\mathscr{L}^{\prime}$ be Lagrangian densities (of the
same order) determining the same action functional (i.e., $\mathscr{L}^{\prime
}=$ $\mathscr{L}+\overline{d}\eta$ for some $\eta\in\overline{\Lambda}{}%
^{n-1}$), and $\mathcal{P}$ and $\mathcal{P}^{\prime}$ the corresponding
generalized HJ\ problems. Then $(\nabla,T)$ is a solution of $\mathcal{P}$ iff
$(\nabla,T-d^{V}\!\eta)$ is a solution of $\mathcal{P}^{\prime}$.
\end{corollary}

\begin{proof}
Let $\vartheta$ be a Legendre form for $\mathscr{L}$. Then $\vartheta^{\prime
}:=\vartheta+d^{V}\!\eta$ is a Legendre form for $\mathscr{L}^{\prime}$. Since,
for $T\in\mathscr{C}\Lambda^{1}\otimes\Lambda^{n-1}$, $T-\vartheta=T-d^{V}\!%
\eta+\vartheta^{\prime}$, trivially, the assertion immediately follows from
Theorem \ref{13}, point \ref{13,3}.
\end{proof}

The above corollary basically states that \emph{the generalized HJ problem is
independent of the Lagrangian density in the class of those determining the
same action functional}.

\begin{remark}
Recall the example in the introduction. Among HJ subdiffieties of a system of
regular, ordinary, $2$nd-order EL equations there are distinguished ones:
namely, those determined (as in the introduction) by solutions of the standard
HJ Eq.{} (\ref{14}). It is well-known that, in its turn, Eq.{}
(\ref{14}) can be geometrically interpreted as follows. Consider the map
\[
T:\mathbb{R}\times\mathbb{R}^{n}\in(t,x)\longmapsto(t,T_{i}(t,x)dx^{i}%
)\in\mathbb{R}\times T^{\ast}\mathbb{R}^{n}.
\]
Then
\[
T_{i}=\dfrac{\partial S}{\partial x^{i}},
\]
with $S=S(x,t)$ a solution of (\ref{14}), iff $\operatorname{im}T$ is an
isotropic submanifold with respect to the presymplectic structure
\[
\Omega_{0}:=dp_{i}\wedge dx^{i}-dH\wedge dt
\]
on $\mathbb{R}\times T^{\ast}\mathbb{R}^{n}$, i.e.,
\begin{equation}
T^{\ast}\Omega_{0}=0. \label{23}%
\end{equation}
It is natural to wonder whether these considerations can be generalized to the
field theoretic setting. A first guess would be to consider the equation
\begin{equation}
(\nabla,T)^{\ast}\omega=T^{\ast}\omega_{0}=0,\quad\operatorname{im}%
(\nabla,T)\subset\mathscr{P} \label{HJE2}%
\end{equation}
as the natural field theoretic generalization of (\ref{23}). However, we think
that this point of view (which is taken in \cite{v10}; see also \cite{dL...08}%
) is not completely satisfactory for the following reasons.

Let $\vartheta$ be a Legendre form. Consider
\[
\Omega:=i_{\mathscr{E}_{EL}}^{\ast}(d^{V}\!\vartheta)\in\mathscr{C}\Lambda
^{2}(\mathscr{E}_{EL})\otimes\overline{\Lambda}{}^{n-1}(\mathscr{E}_{EL}).
\]
In view of the first variation formula (\ref{Legendre}), $\overline{d}\Omega
=0$. The horizontal cohomology class
\[
\boldsymbol{\omega}:=[\Omega]\in H^{n-1}(\mathscr{C}\Lambda^{2}%
(\mathscr{E}_{EL})\otimes\overline{\Lambda}(\mathscr{E}_{EL}),\overline{d})
\]
does only depend on the action functional and it is naturally interpreted as a
(pre)symplectic form in the so called covariant phase space, i.e., the space
of solutions of the EL equations (see, for instance, \cite{v09} for details).
This functional symplectic structure should be understood as the fundamental
symplectic structure in field theory. For instance, it is at the basis of the
BV formalism \cite{ht92}.

Now, let $(\nabla,T)$ be a solution of Eq.{} (\ref{HJE2}) and $\vartheta$ a
Legendre form such that $(T-\vartheta)|_{\nabla_{\lbrack\infty]}}=0$. It is
easy to see that $d^{V}\!T=0$. Therefore,
\[
\nabla_{\lbrack\infty]}^{\ast}\boldsymbol{\omega}=\nabla_{\lbrack\infty
]}^{\ast}[\Omega]=[d^{V}\!\nabla_{\lbrack\infty]}^{\ast}\vartheta]=[d^{V}\!T]=0,
\]
where the last two are cohomology classes in $H^{n-1}(V\!\Lambda^{2}(J^{k}%
,\pi_{k})\otimes\Lambda_{\bullet}^{\bullet}(J^{k},\pi_{k}),d^{V})$. We
conclude that $\mathscr{Y}_{\infty}^{\nabla}$ is an isotropic subdiffiety of
the \textquotedblleft(pre)symplectic diffiety\textquotedblright%
\ $(\mathscr{E}_{EL},\boldsymbol{\omega})$. Unfortunately, this is not a
special feature of solutions of (\ref{HJE2}). Namely, for $n>1$, \emph{every}
finite dimensional subdiffiety of $(\mathscr{E}_{EL},\boldsymbol{\omega})$ is
isotropic, modulo topological obstructions. Indeed, the horizontal de Rham
complex of a finite dimensional diffiety is locally acyclic in positive degree
(see, for instance, the final comment of Section \ref{SecNotConv}). On the
other hand, for $n=1$, solutions of (\ref{HJE2}) may be effectively
carachterized as those solutions of the generalized HJ problem defining
isotropic subdiffieties of $(\mathscr{E}_{EL},\boldsymbol{\omega})$. Because
of this difference between the $n>1$ and the $n=1$ cases, we think that
Eq.{} (\ref{HJE2}) is not as fundamental in field theory as Equation
(\ref{23}) in mechanics and there possibly exists a different, more
fundamental, field theoretic version of Eq.{} (\ref{23}). This possibility
will be explored elsewhere.
\end{remark}

\section{A Final Example\label{SecFinEx}}

The \textquotedblleft good\textquotedblright\ Boussinesq equation
\[
u_{tt}=(u+u^{2}+u_{xx})_{xx}.
\]
is obviously covered by the system of evolutionary equations
\begin{equation}
\left\{
\begin{array}
[c]{l}%
v_{t}=u+u^{2}+u_{xx}\\
u_{t}=v_{xx}%
\end{array}
\right.  ,\label{17}%
\end{equation}
which can be understood as a submanifold in $J^{2}\pi$ with $\pi$ the trivial
bundle $\pi:\mathbb{R}^{4}\ni(t,x,u,v)\longmapsto(t,x)\in\mathbb{R}^{2}$.
Further, Equations (\ref{17}) are the EL equations determined by the action
$\int Ldtdx$ with \cite{ch08}%
\[
L=\dfrac{1}{2}(u_{x}^{2}+v_{x}^{2}+vu_{t}-uv_{t})+\dfrac{1}{3}u^{3}+\dfrac
{1}{2}u^{2}\in C^{\infty}(J^{1}\pi).
\]
The Lagrangian density $L$ is singular in the sense that
\[
\det\left(
\begin{array}
[c]{cc}%
\dfrac{\partial L}{\partial u_{i}\partial u_{j}} & \dfrac{\partial L}{\partial
u_{i}\partial v_{j}}\\
\dfrac{\partial L}{\partial v_{i}\partial u_{j}} & \dfrac{\partial L}{\partial
v_{i}\partial v_{j}}%
\end{array}
\right)  _{i,j=t,x}=\det\left(
\begin{array}
[c]{cccc}%
0 & 0 & 0 & 0\\
0 & 1 & 0 & 0\\
0 & 0 & 0 & 0\\
0 & 0 & 0 & 1
\end{array}
\right)  =0.
\]
Denote by $p^{x},p^{t},q^{x},q^{t}$ momentum coordinates in $J^{\dag}\pi$
associated with $u,v$ respectively. Then
\[
\omega=-dp^{x}\wedge du\wedge dt+dp^{t}\wedge du\wedge dx-dq^{x}\wedge
dv\!\wedge dt+dq^{t}\wedge dv\!\wedge dx-dH\wedge dt\wedge dx,
\]
with%
\[
H=\dfrac{1}{2}(u_{x}^{2}+v_{x}^{2})-\dfrac{1}{3}u^{3}-\dfrac{1}{2}u^{2}.
\]
It follows that
\[
\mathscr{P}:\left\{
\begin{array}
[c]{l}%
p^{x}-u_{x}=0\\
p^{t}+\dfrac{1}{2}v=0\\
q^{x}-v_{x}=0\\
q^{t}-\dfrac{1}{2}u=0
\end{array}
\right.  ,
\]
and
\[
\mathscr{P}_{0}:\left\{
\begin{array}
[c]{l}%
p^{t}+\dfrac{1}{2}v=0\\
q^{t}-\dfrac{1}{2}u=0
\end{array}
\right.
\]
Consequently, fibers of $\mathscr{P}_{0}\longrightarrow J^{0}\pi
=\mathbb{R}^{4}$ are coordinatized by $p^{x},p^{t}$ only and
\[
\omega_{0}=du\wedge dv\!\wedge dx-dp^{x}\wedge du\wedge dt-dq^{x}\wedge
dv\!\wedge dt-dH_{0}\wedge dt\wedge dx,
\]
with
\[
H_{0}=\dfrac{1}{2}[(p^{x})^{2}+(q^{x})^{2}]-\dfrac{1}{6}u^{3}-\dfrac{1}%
{2}u^{2}%
\]
Let
\[
T=-P^{x}d^{V}\!u\otimes dt+P^{t}d^{V}\!u\otimes dx-Q^{x}d^{V}\!v\otimes
dt+Q^{t}d^{V}\!v\otimes dx
\]
be a section of $J^{\dag}\pi\longrightarrow J^{0}\pi$ and
\[
\nabla=(du-Adx-Bdt)\otimes\dfrac{\partial}{\partial u}+(dv-Cdx-Ddt)\otimes
\dfrac{\partial}{\partial v}%
\]
be a connection in $\pi$. $\nabla$ is flat iff
\[
\nabla_{t}A=\nabla_{x}B,\quad\text{and\quad}\nabla_{t}C=\nabla_{x}D,
\]
where $\nabla_{t}=\partial/\partial t+B\partial/\partial u+D\partial/\partial
v$ and $\nabla_{x}=\partial/\partial x+A\partial/\partial u+C\partial/\partial
v$. Moreover, $\operatorname{im}(\nabla,T)\subset\mathscr{P}$ iff
\[
P^{x}=A,\text{\quad}P^{t}=-\dfrac{1}{2}v,\quad Q^{x}=C,\quad\text{and\quad
}Q^{t}=\dfrac{1}{2}u.
\]
In this case $\operatorname{im}T\subset\mathscr{P}_{0}$ and
\begin{align*}
(\nabla,T)^{\ast}(\omega) &  =T^{\ast}(\omega_{0})\\
&  =du\wedge dv\!\wedge dx+(A_{v}-C_{u})du\wedge dv\!\wedge Dt\\
&  -(A_{x}+AA_{u}+CC_{u}-u^{2}-u)du\wedge dt\wedge dx\\
&  -(C_{x}+AA_{v}+CC_{v})dv\!\wedge dt\wedge dx.
\end{align*}
Notice that there is no $T$ such that $T^{\ast}(\omega_{0})=0$. Finally
\[
i^{1,n}(\nabla)T^{\ast}(\omega_{0})=[(D-\nabla_{x}A+u^{2}+u)d^{V}
\!u-(B+\nabla_{x}C)d^{V}\!v]\otimes dt\wedge dx,
\]
so $i^{1,n}(\nabla)T^{\ast}(\omega_{0})=0$ iff
\begin{equation}
\left\{
\begin{array}
[c]{l}%
D-\nabla_{x}A+u^{2}+u=0\\
B+\nabla_{x}C=0
\end{array}
\right.  ,\label{18}%
\end{equation}
which are precisely the $0$th generalized HJ equations for (\ref{17}). Let us
search for solutions of the form
\[
A=A(u),\quad C=C(u),\quad B=-cA,\quad D=-cC,\quad c=\mathrm{const}.
\]
Notice that, in these hypotheses, $\nabla$ is identically flat. Equations
(\ref{18}) reduce to
\begin{align*}
AA_{u} &  =cC+u^{2}+u\\
AC_{u} &  =-cA
\end{align*}
and for $A\neq0$ we find
\[
C=-cu+a,\quad A^{2}=\dfrac{1}{3}u^{3}+\dfrac{1}{2}(1-c^{2})u^{2}+au+b,
\]
$a,b$ being integration constants. We conclude that a (local) solution of the
generalized HJ problem is $(\nabla,T)$ with
\begin{gather*}
A=\sqrt{\dfrac{1}{3}u^{3}+\dfrac{1}{2}(1-c^{2})u^{2}+au+b},\quad B=-cA,\quad
C=-cu+a,\quad D=-cC\\
P^{x}=A,\quad P^{t}=-\dfrac{1}{2}v,\quad Q^{x}=C,\quad Q^{t}=\dfrac{1}{2}u,
\end{gather*}
and the system
\begin{equation}
\left\{
\begin{array}
[c]{l}%
u_{t}=-c\sqrt{\dfrac{1}{3}u^{3}+\dfrac{1}{2}(1-c^{2})u^{2}+au+b}\\
u_{x}=\sqrt{\dfrac{1}{3}u^{3}+\dfrac{1}{2}(1-c^{2})u^{2}+au+b}\\
v_{t}=c^{2}u-ca\\
v_{x}=-cu+a
\end{array}
\right.  ,\label{19}%
\end{equation}
(locally) correspond to an HJ subdiffiety of (\ref{17}). In the case $a=b=0$,
$0<c^{2}<1$ solutions of (\ref{19}) are
\begin{align}
u &  =-\dfrac{3}{2}(1-c^{2})\mathrm{sech}^{2}\left[  \dfrac{1}{4}%
\sqrt{2(1-c^{2})}(x-x_{0}-ct)\right]  ,\label{20}\\
v &  =v_{0}-3\sqrt{2(1-c^{2})}\mathrm{tanh}\left[  \dfrac{1}{4}\sqrt
{2(1-c^{2})}(x-x_{0}-ct)\right]  ,\nonumber
\end{align}
with $x_{0},v_{0}$ integration constants. They are solutions of the EL
equations (\ref{17}). In particular (\ref{20}) are well-known
\textquotedblleft\emph{travelling wave}\textquotedblright\ solutions of the
\textquotedblleft good\textquotedblright\ Boussinesq equation.


\begin{thebibliography}{99}                                                                                               %


\bibitem {v01}A.{} M.{} Vinogradov, Cohomological Analysis of Partial
Differential Equations and Secondary Calculus, \emph{Transl.{} Math.{} Mon.{}
}\textbf{204}, Amer.{} Math.{} Soc., Providence, 2001.

\bibitem {v98}A.{} M.{} Vinogradov, Introduction to Secondary Calculus, in
\emph{Secondary Calculus and Cohomological Physics}, M.{} Henneaux, I.{} S.{}
Krasil'shchik, and A.{} M.{} Vinogradov (Eds.), Contemp.{} Math.{}
\textbf{219}, Amer.{} Math.{} Soc., Providence, 1998, pp.{} 241--272.

\bibitem {c...06}J.{} F.{} Cari\~{n}ena, et al., Geometric Hamilton-Jacobi
Theory, \emph{Int.{} J.{} Geom.{} Meth.{} Mod.{} Phys.{} }\textbf{3} (2006)
1417--1458; e-print: arXiv:math-ph/0604063.

\bibitem {c...09}J.{} F.{} Cari\~{n}ena, et al., Hamilton-Jacobi Theory and
the Evolution Operator, in Mathematical Physics and Field Theory. Julio Abad,
in memoriam, M.{} Asorey, J.{} V.{} Garc\'{\i}a Esteve, M.{} F.{} Ra\~{n}ada
and J.{} Sesma (Eds.), Prensas Universitarias de Zaragoza, Zaragoza, 2009,
pp.{} 177--186; e-print: arXiv:0907.1039.

\bibitem {r73}H.{} Rund, The Hamilton-Jacobi Theory in the Calculus of
Variations, Robert E.{} Krieger Publ.{} Co., Nuntington, N.Y., 1973.

\bibitem {c...10}J.{} F.{} Cari\~{n}ena, et al., Geometric Hamilton-Jacobi
Theory for Nonholonomic Dynamical Systems, \emph{Int.{} J.{} Geom.{} Meth.{}
Mod.{} Phys.{} } \textbf{7} (2010) 431--454; e-print: arXiv:0908.2453.

\bibitem {v10}L.{} Vitagliano, The Hamilton-Jacobi Formalism for Higher Order
Field Theories, \emph{Int.{} J.{} Geom.{} Meth.{} Mod.{} Phys.{} } \textbf{7}
(2010) 1413--1436; e-print: arXiv:1003.5236.

\bibitem {v10b}L.{} Vitagliano, The Lagrangian-Hamiltonian Formalism for
Higher Order Field Theories, \emph{J.{} Geom.{} Phys. }\textbf{60} (2010)
857--873; e-print: arXiv:0905.4580.

\bibitem {r05}N. Rom\'{a}n-Roy, Multisymplectic Lagrangian and Hamiltonian
Formalism of First-Order Classical Field Theories, \emph{SIGMA} \textbf{5}
(2009) 100--125; e-print: arXiv:math-ph/0506022.

\bibitem {v09b}L. Vitagliano, Partial Differential Hamiltonian Systems, (2009)
\emph{submitted for publication}; e-print: arXiv:0903.4528.

\bibitem {b...99}A. V. Bocharov et al., Symmetries and Conservation Laws for
Differential Equations of Mathematical Physics, \emph{Transl. Math. Mon.
}\textbf{182}, Amer. Math. Soc., Providence, 1999.

\bibitem {v84}A.{} M.{} Vinogradov, The $\mathscr{C}$-Spectral Sequence,
Lagrangian Formalism and Conservation Laws I, II, \emph{J.{} Math.{} Anal.{}
Appl.{} }\textbf{100} (1984) 1--129.

\bibitem {a92}I.{} M.{} Anderson, Introduction to the Variational Bicomplex,
in \emph{Math.{} Aspects of Classical Field Theory}, M.{} Gotay, J.{} E.{}
Marsden, and V.{} E.{} Moncrief (Eds.), \emph{Contemp.{} Math.{} }%
\textbf{132}, Amer.{} Math.{} Soc., Providence, 1992, pp.{} 51--73.

\bibitem {av04}R.{} J.{} Alonso--Blanco, and A.{} M.{} Vinogradov, Green
Formula and Legendre Transformation, \emph{Acta Appl.{} Math.{} }\textbf{83},
n$%
{{}^\circ}%
$ 1--2 (2004) 149--166.

\bibitem {l19}G.{} H.{} Livens, On Hamilton's Principle and the Modified
Function in Analytical Dynamics, \emph{Proc.{} Roy.{} Soc.{} Edinb.}
\textbf{39} (1919) 113

\bibitem {d35}Th. de Donder, Th\'{e}orie Invariantive du Calcul des
Variations, Gauthier Villars, Paris, 1935, pp. 95--108.

\bibitem {v09}L.{} Vitagliano, Secondary Calculus and the Covariant Phase
Space, \emph{J.{} Geom.{} Phys.{} }\textbf{59} (2009) 426--447.

\bibitem {ht92}M.{} Henneaux, and C.{} Teitelboim, Quantization of Gauge
Systems, Princeton Univ.{} Press, Princeton, 1992.

\bibitem {dL...08}M.{} de Leon, J.{} C.{} Marrero, D.{} Mart\'{\i}n de Diego,
A Geometric Hamilton-Jacobi Theory for Classical Field Theories; e-print: arXiv:0801.1181.

\bibitem {ch08}J.{} Chen, H.{} Liu, Derivation of Lagrangian Density for the
\textquotedblleft Good\textquotedblright\ Boussinesq Equation and
Multisymplectic Discretizations, \emph{Appl.{} Math.{} Comp.{} }\textbf{204}
(2008) 58--62.
\end{thebibliography}
\end{document}